\documentclass[a4paper,11pt]{article}
\usepackage{amsmath,amsthm,amssymb}
\newtheorem*{theorem*}{Theorem}
\newtheorem*{corollary*}{Corollary}
\newtheorem*{proposition*}{Proposition}
\newtheorem*{remark*}{Remark}
\newtheorem*{lemma*}{Lemma}
\newtheorem*{example*}{Example}

\newtheorem{theorem}{Theorem}[section]

\newtheorem{proposition}[theorem]{Proposition}
\newtheorem{corollary}[theorem]{Corollary}
\newtheorem{lemma}[theorem]{Lemma}
\newtheorem{example}[theorem]{Example}
\newtheorem{remark}[theorem]{Remark}

\makeatletter
\@addtoreset{equation}{section}
\makeatother

\newcommand{\be}{\mathbf{e}}
\newcommand{\bk}{\mathbf{k}}
\newcommand{\bl}{\mathbf{l}}
\newcommand{\bm}{\mathbf{m}}
\newcommand{\bn}{\mathbf{n}}

\newcommand{\BZ}{\mathbb{Z}}
\newcommand{\BR}{\mathbb{R}}
\newcommand{\BC}{\mathbb{C}}

\newcommand{\BO}{\mathbb{O}}

\newcommand{\cF}{\mathcal{F}}
\newcommand{\cH}{\mathcal{H}}

\newcommand{\cO}{\mathcal{O}}
\newcommand{\cP}{\mathcal{P}}

\newcommand{\cU}{\mathcal{U}}
\newcommand{\fa}{\mathfrak{a}}
\newcommand{\fg}{\mathfrak{g}}

\newcommand{\fk}{\mathfrak{k}}
\newcommand{\fl}{\mathfrak{l}}

\newcommand{\fn}{\mathfrak{n}}
\newcommand{\fp}{\mathfrak{p}}
\newcommand{\ft}{\mathfrak{t}}

\newcommand{\rC}{\mathrm{C}}

\newcommand{\rK}{\mathrm{K}}
\newcommand{\rL}{\mathrm{L}}
\newcommand{\rR}{\mathrm{R}}

\newcommand{\tr}{\operatorname{tr}}
\newcommand{\eqspace}{\mathrel{\phantom{=}}}

\newcommand{\Hom}{\operatorname{Hom}}
\newcommand{\End}{\operatorname{End}}
\newcommand{\Det}{\operatorname{Det}}
\newcommand{\Tr}{\operatorname{Tr}}
\newcommand{\Sym}{\operatorname{Sym}}

\newcommand{\Alt}{\operatorname{Alt}}
\newcommand{\Herm}{\operatorname{Herm}}

\newcommand{\rank}{\operatorname{rank}}

\newcommand{\hotimes}{\mathbin{\hat{\otimes}}}

\newcommand{\hsum}{\sideset{}{^\oplus}\sum}
\renewcommand{\Re}{\operatorname{Re}}
\renewcommand{\Im}{\operatorname{Im}}
\renewcommand{\det}{\operatorname{det}}

\newcommand{\op}{\mathrm{op}}

\newcommand{\dist}{\operatorname{dist}}
\newcommand{\artanh}{\operatorname{artanh}}

\newcommand{\Stab}{\operatorname{Stab}}
\newcommand{\diag}{\operatorname{diag}}
\newcommand{\opm}{\mathrm{m}}
\newcommand{\Int}{\operatorname{Int}}
\newcommand{\Ad}{\operatorname{Ad}}
\newcommand{\delete}[1]{}
\allowdisplaybreaks[3]

\title{Holographic operators for the tensor products \\ 
of the spaces of holomorphic functions \\ 
on Hermitian symmetric spaces of tube type}
\author{Ryosuke Nakahama\thanks{This work was supported by JSPS KAKENHI Grant Number JP25K07057, Japan.}
\thanks{Email: ryosuke.nakahama@ntt.com} \\
\textit{NTT Institute for Fundamental Mathematics,} \\ \textit{Communication Science Laboratories, NTT, Inc.} \\
\textit{3-9-11 Midori-cho, Musashino-shi, Tokyo 180-8585, Japan}}
\date{\today}

\begin{document}

\maketitle

\begin{abstract}
We consider a tensor product of two spaces of holomorphic functions on a Hermitian symmetric space of tube type. 
Then generically this is decomposed into a direct sum of irreducible subrepresentations.
In this manuscript, we construct the intertwining operator (holographic operator) from each irreducible summand to the tensor product as an integral operator.
This gives a generalization of the result by Kobayashi--Pevzner \cite{KP3}. 
We also discuss the expression of holographic operators in terms of infinite-order differential operators. 
\bigskip

\end{abstract}

\section{Introduction}

The purpose of this manuscript is to construct intertwining operators (holographic operators) for tensor products of two spaces of holomorphic functions on a Hermitian symmetric space of tube type. 

First, we consider the upper half plane $\Pi:=\{x\in\BC\mid \Im x>0\}$. We fix $\lambda\in\BC$. 
Then the universal covering group $\widetilde{SL}(2,\BR)$ of $SL(2,\BR)$ acts on the space $\cO(\Pi)=\cO_\lambda(\Pi)$ of holomorphic functions on $\Pi$ by 
\[ \biggl(\tau_\lambda\biggl(\begin{pmatrix}a&b\\c&d\end{pmatrix}^{-1}\biggr)f\biggr)(x):=(cx+d)^{-\lambda}f\biggl(\frac{ax+b}{cx+d}\biggr). \]
The representation $(\tau_\lambda,\cO_\lambda(\Pi))$ is irreducible if $\lambda\notin-\BZ_{\ge 0}$. 
We note that $(cx+d)^{-\lambda}$ is not well-defined on $SL(2,\BR)$ if $\lambda\notin\BZ$, but this is well-defined on the universal covering group. 
Next, for $\lambda,\mu\in\BC$, we consider the tensor product representation $\cO_\lambda(\Pi)\hotimes\cO_\mu(\Pi)$. 
If $\lambda,\mu>0$, then this is decomposed into the direct sum of subrepresentations as 
\[ \cO_\lambda(\Pi)\hotimes\cO_\mu(\Pi)\simeq\hsum_{l\in\BZ_{\ge 0}}\cO_{\lambda+\mu+2l}(\Pi). \]
According to this decomposition, the intertwining operator onto each subrepresentation (symmetry breaking operator, in the terminology of \cite{KP1, KP2}) is given by the following Rankin--Cohen bracket, which was originally introduced in the context of modular forms. 

\begin{theorem}[Rankin--Cohen \cite{C, R}]
Let $\lambda,\mu\in\BC$, $l\in\BZ_{\ge 0}$. Then the map 
\begin{gather*}
\cF^{\lambda,\mu}_{l\downarrow}\colon\cO_\lambda(\Pi)\hotimes\cO_\mu(\Pi)\longrightarrow\cO_{\lambda+\mu+2l}(\Pi), \\
(\cF^{\lambda,\mu}_{l\downarrow} f)(z):=\sum_{j=0}^l(-1)^j\frac{(\lambda+l-j)_j(\mu+j)_{l-j}}{j!(l-j)!}\frac{\partial^l f}{\partial x^{l-j}\partial y^j}(x,y)\biggr|_{x=y=z} 
\end{gather*}
intertwines the $\widetilde{SL}(2,\BR)$-action. 
\end{theorem}

Here, $(\lambda)_j:=\lambda(\lambda+1)(\lambda+2)\cdots(\lambda+j-1)$. 
On the other hand, the intertwining operator of opposite direction (holographic operator) has recently been constructed as an integral operator by Kobayashi--Pevzner. 

\begin{theorem}[Kobayashi--Pevzner \cite{KP3}]\label{thm_intro}
Let $\lambda,\mu\in\BC$, $l\in\BZ_{\ge 0}$, $\Re\lambda, \Re\mu>-l$. Then the map 
\begin{gather*}
\cF^{\lambda,\mu}_{l\uparrow}\colon\cO_{\lambda+\mu+2l}(\Pi)\longrightarrow\cO_\lambda(\Pi)\hotimes\cO_\mu(\Pi), \\
\begin{split}
(\cF^{\lambda,\mu}_{l\uparrow} f)(x,y):=\frac{(x-y)^l}{2^{\lambda+\mu+2l-1}l!}\int_{-1}^1 f\biggl(\frac{(y-x)z+(x+y)}{2}\biggr)(1-z)^{\lambda+l-1}(1+z)^{\mu+l-1}\,dz
\end{split}
\end{gather*}
intertwines the $\widetilde{SL}(2,\BR)$-action. 
\end{theorem}

By putting $\frac{1}{2}((y-x)z+(x+y))=:w$, this map is rewritten as 
\begin{align*}
(\cF^{\lambda,\mu}_{l\uparrow} f)(x,y)
&=\frac{1}{l!}{(x-y)^{-\lambda-\mu-l+1}}\int_{[y,x]} f(w)(w-y)^{\lambda+l-1}(x-w)^{\mu+l-1}\,dw.
\end{align*}
In this manuscript, we consider a generalization of the result by Kobayashi--Pevzner for general simple Hermitian symmetric spaces of tube type, namely for $D\simeq G/K= Sp(r,\BR)/U(r)$, $SU(r,r)/S(U(r)\times U(r))$, $SO^*(4r)/U(2r)$, $SO_0(2,n)/SO(2)\times SO(n)$ and $E_{7(-25)}/U(1)\times E_6$, realized as bounded symmetric domains in complex simple Jordan algebras $\fp^+=\Sym(r,\BC)$, $M(r,\BC)$, $\Alt(2r,\BC)$, $\BC^n$, and $\Herm(3,\BO)^\BC$ respectively. 
Our proof is based on a method different from Kobayashi--Pevzner's one, and gives a new insight applicable to symmetric spaces of higher rank. 

Differential symmetry breaking operators are regarded as generalizations of Rankin--Cohen brackets \cite{C, R} for modular forms and Juhl's operators \cite{Ju} for conformal geometry. 
The construction of differential symmetry breaking operators for the symmetric pairs $(G\times G,\allowbreak\diag(G))$ is studied by, e.g, \cite{BCK, Cl, OR, P, PZ, vDP} and the author \cite[Section 5.2]{N1}, \cite[Section 8]{N3}, 
and that for more general symmetric pairs $(G,G')$ is studied by, e.g., \cite{DL, IKO, KKP, KOSS, KP1, KP2, Ku, PV1, PV2} and the author \cite{N1, N2, N4}. 
Especially, by \cite{KP1, KP2}, symmetry breaking operators for symmetric pairs $(G,G')$ of holomorphic type and for holomorphic discrete series representations of $G$ (or more generally, representations on the spaces of holomorphic functions on Hermitian symmetric spaces) are always given by differential operators (localness theorem), and their symbols are characterized as solutions of certain differential equations (F-method). 
The author \cite{N3} also computes the operator norms of symmetry breaking operators and determines the Parseval--Plancherel type formulas for holomorphic discrete series representations of scalar type. 
Intertwining operators for certain non-symmetric pairs are studied by, e.g., \cite{FL, La}, and non-differential intertwining operators for certain symmetric pairs are also studied by, e.g., \cite{DF1, DF2, FO, FW1, FW2, KL, KS1, KS2, MO}. 

On the other hand, the study on holographic operators for holomorphic discrete series representations is initiated by Kobayashi--Pevzner \cite{KP3}. 
In \cite{KP3}, they construct a holographic operator for $(SL(2,\BR)\times SL(2,\BR),\allowbreak SL(2,\BR))$ as an integral operator on a real line segment (Theorem \ref{thm_intro}), and that for $(SO_0(2,n),SO_0(2,n-1))$ as an integral operator on the complex tube domain. 
The author \cite[Section 3.2]{N1} also constructs holographic operators for general symmetric pairs of holomorphic type and general holomorphic discrete series representations, 
as integral operators on complex bounded symmetric domains (see Theorem \ref{thm_int_expr_D} of this article for tensor product cases). 
In addition, once a holographic operator is constructed, then we can also express it as an infinite-order differential operator (see \cite[Section 3.3]{N1} and Theorem \ref{thm_diff_expr}). 

However, this integral operator (Theorem \ref{thm_int_expr_D}) does not contain Theorem \ref{thm_intro} as a special case, since the operator in Theorem \ref{thm_intro} is given as an integral on a totally real submanifold. 
In this article, we construct holographic operators for the tensor products of holomorphic discrete series representations as integral operators on totally real submanifolds, when $G$ is of tube type, and the subrepresentation satisfies a suitable assumption. 
Especially, if we consider the tensor product of two holomorphic discrete series representations of scalar type, then all subrepresentations satisfy this assumption, and these operators completely work. 
This enables us to analyze the representations via the theory of contour integrals. 
Moreover, we can explicitly compute the differential expression of the holographic operators if three representations (two tensored representations and a subrepresentation) are of scalar type. 
We expect that the combination of several expressions of symmetry breaking and holographic operators (differentials, integrals over complex domains, integrals over totally real submanifolds) provides new insights into other areas, e.g., automorphic forms and conformal geometries. 
We note that the proof of Theorem \ref{thm_intro} by \cite{KP3} is performed by transforming the Rankin--Cohen bracket into the $L^2(\BR_{>0})$-picture, and applying the theory of Jacobi polynomials. 
On the other hand, in this paper, our proof is based on a different approach, and does not treat the relation with differential symmetry breaking operators, $L^2$-pictures, and multivariate orthogonal polynomials. 
Further research is required to find another proof of our theorem parallel to \cite{KP3}, which will make the theory more fruitful. 

This manuscript is organized as follows. 
In Section \ref{section_prelim}, we review Jordan algebras and Hermitian symmetric spaces. 
In Section \ref{section_contour}, we give a definition of certain totally real submanifolds in the Hermitian symmetric spaces, 
and in Section \ref{section_construction}, we construct holographic operators as integral operators on the above totally real submanifolds. 
In Section \ref{section_minKtype}, we consider the tensor products of the spaces of scalar-valued holomorphic functions, and for each irreducible subrepresentation, we compute the image of the minimal $K$-type under holographic operators. 
This corresponds to the determination of the normalization of holographic operators. 
Finally, in Section \ref{section_diff}, when the subrepresentation is also of scalar type, 
we prove that the holographic operator is expressed as an infinite-order differential operator, with the symbol given by a multivariate confluent hypergeometric function.

\section{Preliminaries}\label{section_prelim}

In this section, we review Jordan algebras and Hermitian symmetric spaces. 
For details, see, e.g., \cite[Parts III, V]{FKKLR}, \cite{FK}, \cite{Kor}, \cite{L0}, \cite{L}, \cite{Sat}. 

\subsection{Jordan algebras and Kantor--Koecher--Tits construction}

Let $\fn^+$ be a simple Euclidean Jordan algebra, with the multiplication $\circ$, the unit element $e$, and the complexification $\fp^+=\fn^{+\BC}:=\fn^+\otimes\BC$. 
Let $\tr$ and $\det$ be the trace and determinant polynomials on $\fp^+$, and let $(x|y):=\tr(x\circ y)$ be the associated bilinear form. 
Also, let $\overline{\cdot}\colon\fp^+\to\fp^+$ be the complex conjugate with respect to the real form $\fn^+\subset\fp^+$, so that $(\cdot|\overline{\cdot})$ is positive definite on $\fp^+$. 
We also write $\fn^+=\fn^-$, $\fp^+=\fp^-$, and regard $\fn^-$, $\fp^-$ as the dual spaces of $\fn^+$, $\fp^+$ via $(\cdot|\cdot)$. 
For $x,y\in\fp^+$, let $L(x),P(x),D(x,y),B(x,y)\in\End_\BC(\fp^+)$ be the maps given by 
\begin{align*}
L(x)z&:=x\circ z, \\
P(x)z&:=2x\circ(x\circ z)-(x\circ x)\circ z, \\
D(x,y)z&:=2(x\circ(y\circ z)+z\circ(y\circ x)-(x\circ z)\circ y), \\
B(x,y)z&:=z-D(x,y)z+P(x)P(y)z, 
\end{align*}
and let $h(x,y)$, $|x|_\infty$ be the generic norm and the spectral norm on $\fp^+$. 
Throughout the paper, let $\rank\fp^+=r$, $\dim\fp^+=n=r+\frac{d}{2}r(r-1)$. Then $(\cdot|\cdot)$, $\det$, $h$ and $|\cdot|_\infty$ satisfy 
\begin{gather*}
(x|y)=\tr(x\circ y)=\frac{r}{n}\Tr(L(x\circ y))=\frac{r}{2n}\Tr(D(x,y)), \\
\det(x)^{\frac{2n}{r}}=\Det(P(x)), \qquad 
h(x,y)^{\frac{2n}{r}}=\Det(B(x,y)), \qquad
|x|_\infty^2=\biggl|\frac{1}{2}D(x,\overline{x})\biggr|_{\op,\fp^+}, 
\end{gather*}
where $\Tr$, $\Det$ and $|\cdot|_{\op,\fp^+}$ are the usual trace, determinant and the operator norm on $\End_\BC(\fp^+)$. 
In addition, let $\Omega\subset\fn^+$, $T_\Omega\subset\fp^+$, and $D\subset\fp^+$ be the symmetric cone, the tube domain, and the bounded symmetric domain given by 
\begin{align*}
\Omega:\hspace{-3pt}&=(\text{connected component of }\{x\in\fn^+\mid \det(x)>0\} \text{ which contains }e)\subset\fn^+, \\
T_\Omega:\hspace{-3pt}&=\fn^++\sqrt{-1}\Omega\subset\fp^+, \\
D:\hspace{-3pt}&=(\text{connected component of }\{x\in\fp^+\mid h(x,\overline{x})>0\} \text{ which contains }0) \\
&=\{x\in\fp^+\mid |x|_\infty<1\}\subset\fp^+. 
\end{align*}
Then $T_\Omega$ is biholomorphically diffeomorphic to $D$ via the Cayley transform 
\[ g_c.\colon D\overset{\sim}{\longrightarrow} T_\Omega, \qquad x\longmapsto (x+\sqrt{-1}e)\circ(\sqrt{-1}x+e)^{-1}. \]

Next, for $l\in \End_\BC(\fp^+)$, let $\overline{l},{}^t\hspace{-1pt}l\in\End_\BC(\fp^+)$ be the elements satisfying $(lx|y)=(x|{}^t\hspace{-1pt}ly)$, $\overline{lx}=\overline{l}\overline{x}$ for $x,y\in\fp^+$, and let $l^*:=\overline{{}^t\hspace{-1pt}l}$. Let 
\[ \fk^\BC:=\{l\in\End_\BC(\fp^+)\mid D(lx,y)-D(x,{}^t\hspace{-1pt}ly)=lD(x,y)-D(x,y)l\}, \]
and let $\fk,\fl,\fk_\fl^\BC,\fk_\fl\subset\fk^\BC$ be the Lie subalgebras given by 
\begin{align*}
\fk&:=\{l\in\fk^\BC\mid l^*=-l\}, &
\fl&:=\{l\in\fk^\BC\mid \overline{l}=l\}, \\
\fk_\fl^\BC&:=\{l\in\fk^\BC\mid {}^t\hspace{-1pt}l=-l\}, &
\fk_\fl&:=\fk\cap\fl=\fk\cap\fk_\fl^\BC=\fl\cap\fk_\fl^\BC. 
\end{align*}
We construct a Lie algebra $\fg^\BC$ via the Kantor--Koecher--Tits construction by 
\[ \fg^\BC:=\fp^+\oplus\fk^\BC\oplus\fp^- \]
equipped with the Lie bracket 
\[ [(x,k,y),(z,l,w)]:=(kz-lx,[k,l]+D(x,w)-D(z,w),-{}^t\hspace{-1pt}kw+{}^t\hspace{-1pt}ly), \]
and extend the bilinear form $(\cdot|\cdot)\colon\fp^+\times\fp^-\to\BC$ to the invariant bilinear form $(\cdot|\cdot)\colon\fg^\BC\times\fg^\BC\to\BC$. 
We fix a connected complex Lie group $G^\BC$ corresponding to $\fg^\BC$. 
Next, let $\fg,{}^c\fg,{}^c\fk\subset\fg^\BC$ be the Lie subalgebras, and $\fp,{}^c\fp\subset\fg^\BC$ be the subspaces given by 
\begin{alignat*}{2}
\fg&:=\{(x,k,\overline{x})\mid x\in\fp^+,\; k\in\fk\}&&\subset\fg^\BC, \\
\fp&:=\{(x,0,\overline{x})\mid x\in\fp^+\}&&\subset\fg, \\
{}^c\fg&:=\fn^+\oplus\fl\oplus\fn^-&&\subset\fg^\BC, \\
{}^c\fk&:=\{(x,k,-x)\mid x\in\fn^+,\; k\in\fk_\fl\}&&\subset{}^c\fg, \\
{}^c\fp&:=\{(x,L(y),x)\mid x,y\in\fn^+\}&&\subset{}^c\fg, 
\end{alignat*}
so that $\fg=\fk\oplus\fp$, ${}^c\fg={}^c\fk\oplus{}^c\fp$ holds, and $\fg,{}^c\fg$ are isomorphic via the Cayley transform 
\[ \Ad(g_c):=\Ad\biggl(\exp\biggl(\frac{\pi\sqrt{-1}}{4}(e,0,e)\biggr)\biggr)\colon\fg\overset{\sim}{\longrightarrow} {}^c\fg.  \]
Let $G,{}^c\hspace{-1pt}G,K,{}^c\hspace{-1pt}K, P^\pm, N^\pm\subset G^\BC$ be the connected closed subgroups corresponding to $\fg,{}^c\fg,\fk,{}^c\fk,\fp^\pm,\fn^\pm$ respectively. 
Also, let $L:=K^\BC\cap {}^c\hspace{-1pt}G$, $K_L:=K\cap L$, $K_L^\BC:=K_L\exp(\sqrt{-1}\fk_\fl)\allowbreak\subset K^\BC$. 
Then $L$ acts on $\Omega$, $K_L^\BC$ acts on $\fp^+$ as Jordan algebra automorphisms, and we have 
\[ \Ad_{G^\BC}(K^\BC)|_{\fp^+}\cap GL_\BR(\fn^+)=\Ad_{G^\BC}(L)|_{\fp^+}\cup-\Ad_{G^\BC}(L)|_{\fp^+}\subset GL_\BC(\fp^+), \]
where an element of $-\Ad_{G^\BC}(L)|_{\fp^+}$ maps $\Omega$ to $-\Omega$ (see \cite[Proposition VIII.2.8]{FK}). 
We abbreviate $\Ad(l)x=:lx$, $\Ad(l^{-1})y\allowbreak=:{}^t\hspace{-1pt}ly$ for $l\in K^\BC$, $x\in\fp^+$, $y\in\fp^-$ if there is no confusion. Also, let 
\[ \fp^{+\times}:=\{x\in\fp^+\mid \det(x)\ne 0\}, \]
and let $\widetilde{K}^\BC, \widetilde{\fp}^{+\times}$ be the universal covering spaces of $K^\BC, \fp^{+\times}$ respectively. We extend the maps 
\begin{align*}
P\colon \fp^{+\times}\longrightarrow \Ad_{G^\BC}(K^\BC)|_{\fp^+}\subset GL_\BC(\fp^+), \qquad B\colon D\times D\longrightarrow \Ad_{G^\BC}(K^\BC)|_{\fp^+}\subset GL_\BC(\fp^+)
\end{align*}
to the maps between the universal covering spaces 
\begin{align*}
P\colon \widetilde{\fp}^{+\times}\longrightarrow \widetilde{K}^\BC, \qquad B\colon D\times D\longrightarrow \widetilde{K}^\BC, 
\end{align*}
denoted by the same symbols.

\subsection{Jordan frames and restricted root systems}\label{subsection_root}

We fix a Jordan frame $\{e_1,\ldots,e_r\}\subset\fn^+$, namely, a maximal set of primitive idempotents satisfying $D(e_i,e_j)=0$ for $i\ne j$, and let $h_j:=D(e_j,e_j)\in\fl\subset\fk^\BC$. We set 
\begin{alignat*}{2}
\fa_\fl&:=\{a_1h_1+\cdots+a_rh_r\mid a_j\in\BR\}&&\subset\fl\subset\fk^\BC, \\
\fa_\fl^\BC&:=\{a_1h_1+\cdots+a_rh_r\mid a_j\in\BC\}&&\subset\fk^\BC, 
\end{alignat*}
and we take a Cartan subalgebra $\ft^\BC\subset\fk^\BC$ which contains $\fa_\fl^\BC$. 
We define the elements $\gamma_j\in(\fa_\fl^\BC)^\vee$ of the dual space by $\gamma_i(h_j)=2\delta_{ij}$, and extend to $\gamma_j\in(\ft^\BC)^\vee$, denoted by the same symbols. Then 
\[ d\chi:=\frac{1}{2}(\gamma_1+\cdots+\gamma_r)\in(\ft^\BC)^\vee \]
becomes a character of $\fk^\BC$. Let $\chi$ be the character of $K^\BC$ corresponding to $d\chi$. Then for $x\in\fp^+$, we have 
\[ d\chi(D(x,y))=(x|y), \qquad \chi(P(x))=\det(x), \qquad \chi(B(x,y))=h(x,y). \]
We take a positive restricted root system $\Sigma_+(\fg^\BC,\fa_\fl^\BC)\subset \Sigma(\fg^\BC,\fa_\fl^\BC)\subset(\fa_\fl^\BC)^\vee$ as 
\[ \Sigma_+(\fg^\BC,\fa_\fl^\BC):=\biggl\{\frac{1}{2}(\gamma_i-\gamma_j)\biggm| 1\le i<j\le r\biggr\}\cup\biggl\{\frac{1}{2}(\gamma_i+\gamma_j)\biggm|1\le i\le j\le r\biggr\}, \]
and we take a positive root system $\Delta_+(\fg^\BC,\ft^\BC)\subset \Delta(\fg^\BC,\ft^\BC)\subset(\ft^\BC)^\vee$ compatible with $\Sigma_+(\fg^\BC,\fa_\fl^\BC)$. 

Next, we consider the restricted root subspaces of $\fl$, 
\begin{gather*}
\fl_{ij}:=\biggl\{l\in\fl\biggm| [h_k,l]=\frac{\delta_{ik}-\delta_{jk}}{2}l,\; 1\le k\le r\biggr\} \qquad (1\le i,j\le r), \\
\fn_\fl:=\bigoplus_{1\le i<j\le r} \fl_{ij}, \qquad \fn_\fl^-:=\bigoplus_{1\le j<i\le r}\fl_{ij}. 
\end{gather*}
Let $A_L, N_L, N_L^-\subset L$ be the connected closed subgroups corresponding to $\fa_\fl, \fn_\fl, \fn_\fl^-$ respectively, 
and let $M_L:=Z_{K_L}(\fa_\fl)$ be the centralizer of $\fa_\fl$ in $K_L$. 
Then $M_LA_LN_L, M_LA_LN_L^-\subset L$ are minimal parabolic subgroups.

\subsection{Representations on the spaces of holomorphic functions}

We consider the action of $G^\BC$ on $G^\BC/K^\BC P^-$. Since $P^+K^\BC P^-/K^\BC P^-\subset G^\BC/K^\BC P^-$ is open dense, for $g\in G^\BC$, $x\in\fp^+$, if $g\exp(x)\in P^+K^\BC P^-$, then we write 
\[ g\exp(x)=\exp(\pi^+(g,x))\kappa(g,x)\exp(\pi^-(g,x)) \]
with $\pi^\pm(g,x)\in\fp^\pm$, $\kappa(g,x)\in K^\BC$. Then $\pi^+$ gives the birational action of $G^\BC$ on $\fp^+$. In the following, we abbreviate $\pi^+(g,x)=:g.x$. 
Also, $\kappa$ satisfies the cocycle condition 
\[ \kappa(g_1g_2,x)=\kappa(g_1,g_2.x)\kappa(g_2,x) \qquad (g_1,g_2\in G^\BC,\; x\in\fp^+). \]
Especially, if $g=\exp((u,0,0))$ with $u\in\fp^+$, we have 
\[ \exp((u,0,0)).x=x+u, \qquad \kappa(\exp((u,0,0)),x)=I, \]
if $g=k\in K^\BC$, we have 
\[ k.x=kx, \qquad \kappa(k,x)=k, \]
and if $g=\exp((0,0,v))$ with $v\in\fp^-$, we have 
\begin{gather*}
\exp((0,0,v)).x=x^{-v}, \qquad \Ad(\kappa(\exp((0,0,v)),x))|_{\fp^+}=B(x,-v)^{-1}\in GL_\BC(\fp^+), 
\end{gather*}
where the \emph{quasi-inverse} $x^v$ is defined as 
\begin{equation}
x^v:=B(x,v)^{-1}(x-P(x)v)=(x^{-1}-v)^{-1}. \label{formula_quasiinv}
\end{equation}
Also, if we put $J:=\exp\bigl(\frac{\pi}{2}(-e,0,e)\bigr)\in G^\BC$, then we have 
\[ J.x=-x^{-1},\qquad \Ad(\kappa(J,x))|_{\fp^+}=P(x)^{-1}\in GL_\BC(\fp^+), \]
and for $g_c:=\exp\bigl(\frac{\pi\sqrt{-1}}{4}(e,0,e)\bigr)$, we have 
\begin{align*}
g_c.x&=(x+\sqrt{-1}e)\circ(\sqrt{-1}x+e)^{-1}, & 
\Ad(\kappa(g_c,x))|_{\fp^+}&=P\bigl(2^{-1/2}(\sqrt{-1}x+e)\bigr)^{-1}, \\
g_c^{-1}.x&=(x-\sqrt{-1}e)\circ(-\sqrt{-1}x+e)^{-1}, & 
\Ad(\kappa(g_c^{-1},x))|_{\fp^+}&=P\bigl(2^{-1/2}(-\sqrt{-1}x+e)\bigr)^{-1}. 
\end{align*}
Moreover, the restriction of $\pi^+$ gives the actions of $G, {}^c\hspace{-1pt}G$ on $D, T_\Omega$ respectively, and 
\[ D\simeq G/K\simeq {}^c\hspace{-1pt}G/{}^c\hspace{-1pt}K\simeq T_\Omega \]
becomes a \emph{Hermitian symmetric space of tube type}. 
Let $\widetilde{G}, {}^c\hspace{-1pt}\widetilde{G}, \widetilde{K}^\BC$ denote the universal covering groups of $G, {}^c\hspace{-1pt}G, K^\BC$ respectively. 
We extend $\kappa\colon G\times D\to K^\BC$, $\kappa\colon {}^c\hspace{-1pt}G\times T_\Omega\to K^\BC$ to the maps between universal covering spaces 
$\kappa\colon \widetilde{G}\times D\to \widetilde{K}^\BC$, $\kappa\colon {}^c\hspace{-1pt}\widetilde{G}\times T_\Omega\to \widetilde{K}^\BC$. 

Next, let $(\tau,V_\tau)$ be an irreducible finite-dimensional representation of $\widetilde{K}^\BC$, with the restricted lowest weight $-\frac{1}{2} \sum_{j=1}^r \lambda_j\gamma_j\in(\fa_\fl^\BC)^\vee$, where $\lambda_j-\lambda_{j+1}\in\BZ_{\ge 0}$. 
Then $\widetilde{G}, {}^c\hspace{-1pt}\widetilde{G}$ act on the spaces of $V_\tau$-valued holomorphic functions $\cO(D,V_\tau), \cO(T_\Omega,V_\tau)$ respectively by 
\[ (\hat{\tau}(g)f)(x):=\tau(\kappa(g^{-1},x))^{-1}f(g^{-1}.x). \]
$\cO(D,V_\tau), \cO(T_\Omega,V_\tau)$ are isomorphic by 
\[ \hat{\tau}(g_c)\colon\cO(D,V_\tau)\longrightarrow \cO(T_\Omega,V_\tau), \qquad 
(\hat{\tau}(g_c)f)(x):=\tau(\kappa(g_c^{-1},x))^{-1}f(g_c^{-1}.x). \]
When $(\tau,V_\tau)$ is a unitary representation of $\widetilde{K}$ with the invariant inner product $(\cdot,\cdot)_\tau$, $\hat{\tau}$ preserves the inner products 
\begin{align*}
\langle f_1,f_2\rangle_{\hat{\tau}}&:=C_{\hat{\tau}}\int_D (\tau(B(x,\overline{x}))^{-1}f_1(x),f_2(x))_\tau h(x,\overline{x})^{-\frac{2n}{r}}\,dx
&& (f_1,f_2\in\cO(D,V_\tau)), \\
\langle f_1,f_2\rangle_{\hat{\tau}}&:=C_{\hat{\tau}}\int_{T_\Omega} (\tau(P(\Im x))^{-1} f_1(x),f_2(x))_\tau \det(\Im x)^{-\frac{2n}{r}}\,dx
&& (f_1,f_2\in\cO(T_\Omega,V_\tau)), 
\end{align*}
where we determine the normalizing constant $C_{\hat{\tau}}$ such that $\Vert v\Vert_{\hat{\tau}}=|v|_\tau$ holds for all constant functions $v\in V_\tau\subset\cO(D,V_\tau)$. 
If the restricted lowest weight $-\frac{1}{2} \sum_{j=1}^r \lambda_j\gamma_j$ of $(\tau,V_\tau)$ satisfies $\lambda_r>\frac{2n}{r}-1$, 
then the corresponding Hilbert spaces $\cH(D,V_\tau)\subset\cO(D,V_\tau)$, $\cH(T_\Omega,V_\tau)\subset\cO(T_\Omega,V_\tau)$ are non-zero, and these are called \emph{holomorphic discrete series representations}. For such cases, the $\widetilde{K}$-finite part of the bounded symmetric domain picture coincides with the space of $V_\tau$-valued polynomials. 
\[ \cO(D,V_\tau)_{\widetilde{K}}=\cH(D,V_\tau)_{\widetilde{K}}=\cP(\fp^+)\otimes V_\tau. \]
On the other hand, this does not hold for general unitary subrepresentations of $\cO(D,V_\tau)$ if it is not a holomorphic discrete series representation. For $\mu\in\BC$, we also write 
\begin{align*}
\cO(D,V_\tau\otimes\chi^{-\mu})&=:\cO_\mu(D,V_\tau), &
\cO(D,\chi^{-\mu})&=:\cO_\mu(D), \\
\cO(T_\Omega,V_\tau\otimes\chi^{-\mu})&=:\cO_\mu(T_\Omega,V_\tau), &
\cO(T_\Omega,\chi^{-\mu})&=:\cO_\mu(T_\Omega). 
\end{align*}
Especially, for $\mu>\frac{2n}{r}-1$, the holomorphic discrete series representations $\cH_\mu(D)\subset\cO_\mu(D)$, $\cH_\mu(T_\Omega)\subset\cO_\mu(T_\Omega)$ are given by the inner products 
\begin{align}
\langle f_1,f_2\rangle_\mu&:=C_\mu\int_Df_1(x)\overline{f_2(x)}h(x,\overline{x})^{\mu-\frac{2n}{r}}\,dx && (f_1,f_2\in\cO_\mu(D)), \label{inner_prod}\\
\langle f_1,f_2\rangle_\mu&:=C_\mu\int_{T_\Omega}f_1(x)\overline{f_2(x)}\det(\Im(x))^{\mu-\frac{2n}{r}}\,dx && (f_1,f_2\in\cO_\mu(T_\Omega)), \notag
\end{align}
where 
\[ C_\mu:=\frac{1}{\pi^n}\prod_{j=1}^r\frac{\Gamma\bigl(\mu-\frac{d}{2}(j-1)\bigr)}{\Gamma\bigl(\mu-\frac{n}{r}-\frac{d}{2}(j-1)\bigr)}. \]

\subsection{Tensor product of the spaces of holomorphic functions}

Let $(\sigma,V_\sigma), (\tau,V_\tau)$ be irreducible finite-dimensional representations of $\widetilde{K}$. 
We consider the tensor product representation $(\hat{\sigma}\hotimes\hat{\tau},\cO(D\times D,V_\sigma\otimes V_\tau))$. 
For a while, we assume that $\cO(D,V_\sigma), \cO(D,V_\tau)$ contain the holomorphic discrete series representations $\cH(D,V_\sigma), \allowbreak\cH(D,V_\tau)$, 
and consider the Hilbert tensor product $\cH(D,V_\sigma)\hotimes\cH(D,V_\tau)$. 
Then this is completely reducible, and each subrepresentation contains a $\fp^+$-null vector. 
Since $\fp^+$ acts on $(\cH(D,V_\sigma)\hotimes\cH(D,V_\tau))_{\widetilde{K}}$ as a directional derivative along $\{(x,x)\mid x\in\fp^+\}$, the space of $\fp^+$-null vectors is given by 
\begin{align*}
(\cH(D,V_\sigma)\hotimes\cH(D,V_\tau))_{\widetilde{K}}^{\fp^+}
=\{f(x-y)\mid f\in\cP(\fp^+)\}\otimes V_\sigma\otimes V_\tau
\simeq \cP(\fp^+)\otimes V_\sigma\otimes V_\tau. 
\end{align*}
Then according to the decomposition of this space under $\widetilde{K}$, the following holds. 

\begin{theorem}[{Kobayashi, \cite[Lemma 8.9]{Kmf1}}]
Suppose $\cH(D,V_\sigma), \cH(D,V_\tau)$ are holomorphic discrete series representations of $\widetilde{G}$. If $\cP(\fp^+)\otimes V_\sigma\otimes V_\tau$ is decomposed under $\widetilde{K}$ as 
\[ \cP(\fp^+)\otimes V_\sigma\otimes V_\tau\simeq\bigoplus_{\rho}V_\rho^{\oplus m(\rho)} \qquad (m(\rho)\in\BZ_{\ge 0}), \]
then $\cH(D,V_\sigma)\hotimes\cH(D,V_\tau)$ is decomposed under $\widetilde{G}$ as 
\[ \cH(D,V_\sigma)\hotimes\cH(D,V_\tau)\simeq\hsum_{\rho}\cH(D,V_\rho)^{\oplus m(\rho)}. \]
\end{theorem}

Especially, suppose $V_\sigma=\chi^{-\lambda}, V_\tau=\chi^{-\mu}$. Let 
\begin{align*}
\BZ_{++}^r&:=\{\bk=(k_1,\ldots,k_r)\in\BZ^r\mid k_1\ge\cdots\ge k_r\ge 0\}, \\ 
\BZ_{+}^r&:=\{\bk=(k_1,\ldots,k_r)\in\BZ^r\mid k_1\ge\cdots\ge k_r\},
\end{align*}
and for $\bk\in\BZ_{+}^r$, let $V_{\bk\gamma}^\vee$ be the irreducible representation of $K^\BC$ with the lowest weight $-(k_1\gamma_1+\cdots+k_r\gamma_r)\in(\ft^\BC)^\vee$, so that 
\begin{equation}
\cP(\fp^+)\simeq \bigoplus_{\bk\in\BZ_{++}^r}V_{\bk\gamma}^\vee, \qquad 
\cP(\fp^+)[\det^{-1}]\simeq \bigoplus_{\bk\in\BZ_{+}^r}V_{\bk\gamma}^\vee \label{formula_HKS}
\end{equation}
hold (see \cite[Theorem XI.2.4]{FK}). Then by the former decomposition, we have the following. 

\begin{theorem}[{Kobayashi, \cite[Theorem 8.4]{Kmf1}}]
Suppose $\lambda,\mu>\frac{2n}{r}-1$. Then $\cH_\lambda(D)\hotimes\cH_\mu(D)$ is decomposed under $\widetilde{G}$ as 
\[ \cH_\lambda(D)\hotimes\cH_\mu(D)\simeq\hsum_{\bk\in\BZ_{++}^r}\cH_{\lambda+\mu}(D,V_{\bk\gamma}^\vee). \]
\end{theorem}

Let $(\rho,V_\rho)$ be a representation of $\widetilde{K}^\BC$ appearing in the decomposition of $\cP(\fp^+)\otimes V_\sigma\otimes V_\tau$. We want to find $\widetilde{G}$-intertwining operators 
\begin{align*}
&\cF_{\rho\uparrow}^{\sigma,\tau}\colon\cO(D,V_\rho)\longrightarrow\cO(D,V_\sigma)\hotimes\cO(D,V_\tau), \\
&\cF_{\rho\downarrow}^{\sigma,\tau}\colon\cO(D,V_\sigma)\hotimes\cO(D,V_\tau) \longrightarrow \cO(D,V_\rho). 
\end{align*}
Such $\cF_{\rho\uparrow}^{\sigma,\tau}$ are called holographic operators, and $\cF_{\rho\downarrow}^{\sigma,\tau}$ are called symmetry breaking operators, in the terminology of \cite{KP1, KP2, KP3}. 
By \cite{N1}, we have integral expressions of these intertwining operators. 

\begin{theorem}[{\cite[Corollary 3.5]{N1}}]\label{thm_int_expr_D}
Suppose $\cH(D,V_\sigma), \cH(D,V_\tau)$ are holomorphic discrete series representations of $\widetilde{G}$. 
Let $(\rho,V_\rho)$ be an irreducible representation of $\widetilde{K}$ such that a non-zero operator-valued polynomial 
\[ \rK(x)\in\bigl(\cP(\fp^+)\otimes\Hom_\BC(V_\rho,V_\sigma\otimes V_\tau)\big)^{\widetilde{K}} \]
exists, and define an operator-valued function $\hat{\rK}(x,y;\overline{w})$ by 
\begin{align*}
\hat{\rK}(x,y;\overline{w}):={}&\bigl(\sigma(B(x,\overline{w}))\otimes\tau(B(y,\overline{w}))\bigr) \rK(x^{\overline{w}}-y^{\overline{w}}) \\
&\in\cO(D\times D\times\overline{D},\Hom_\BC(V_\rho,V_\sigma\otimes V_\tau)),  
\end{align*}
where $x^w$ is the quasi-inverse defined as in (\ref{formula_quasiinv}). Then the linear maps 
\begin{gather*}
\tilde{\cF}^{\sigma,\tau}_{\rho,\rK\uparrow}\colon\cH(D,V_\rho)\longrightarrow \cH(D,V_\sigma)\hotimes\cH(D,V_\tau), \\
(\tilde{\cF}^{\sigma,\tau}_{\rho,\rK\uparrow}f)(x,y):=\langle f,\hat{\rK}(x,y;\overline{\cdot})^*\rangle_{\hat{\rho}}
=C_{\hat{\rho}}\int_D \hat{\rK}(x,y;\overline{w})\rho(B(w,\overline{w}))^{-1}f(w) h(w,\overline{w})^{-\frac{2n}{r}}\,dw, \\
\tilde{\cF}^{\sigma,\tau}_{\rho,\rK\downarrow}\colon\cH(D,V_\sigma)\hotimes\cH(D,V_\tau) \longrightarrow \cH(D,V_\rho), \\
\begin{split}
&(\tilde{\cF}^{\sigma,\tau}_{\rho,\rK\downarrow}f)(w):=\langle f,\hat{\rK}(\cdot,\cdot;\overline{w})\rangle_{\hat{\sigma}\otimes\hat{\tau}} \\
&=C_{\hat{\sigma}}C_{\hat{\tau}}\int_{D\times D} \hat{\rK}(x,y;\overline{w})^* \bigl(\sigma(B(x,\overline{x}))\otimes\tau(B(y,\overline{y}))\bigr)^{-1}f(x,y) h(x,\overline{x})^{-\frac{2n}{r}}h(y,\overline{y})^{-\frac{2n}{r}}\,dxdy
\end{split}
\end{gather*}
intertwine the $\widetilde{G}$-actions. 
\end{theorem}

A holographic operator $\cF^{\sigma,\tau}_{\rho\uparrow}$ is also realized as an infinite-order differential operator, 
and a symmetry breaking operator $\cF^{\sigma,\tau}_{\rho\downarrow}$ is also realized as a finite-order differential operator. 
Let $2D:=\{2x\mid x\in D\}$. 

\begin{theorem}\label{thm_diff_expr}
\begin{enumerate}
\item Suppose that the map 
\[ \cF_{\rho\uparrow}^{\sigma,\tau}=\cF_{\rho\uparrow,w}^{\sigma,\tau}\colon\cO(D,V_\rho)\longrightarrow\cO(D,V_\sigma)\hotimes\cO(D,V_\tau)  \]
intertwines the $\widetilde{G}$-action. We define the functions $F^{\sigma,\tau}_{\rho\uparrow,\rC}(x;\zeta)$, $F^{\sigma,\tau}_{\rho\uparrow,\rL}(x;\zeta)$, 
$F^{\sigma,\tau}_{\rho\uparrow,\rR}(x;\zeta)\allowbreak\in\cO(D\times\fp^-,\Hom_\BC(V_\rho,V_\sigma\otimes V_\tau))$ by 
\begin{align*}
F^{\sigma,\tau}_{\rho\uparrow,\rC}(x;\zeta)&:=(\cF^{\sigma,\tau}_{\rho\uparrow}e^{(w|\zeta)}I_\rho)_w(x,-x), \\
F^{\sigma,\tau}_{\rho\uparrow,\rL}(x;\zeta)&:=(\cF^{\sigma,\tau}_{\rho\uparrow}e^{(w|\zeta)}I_\rho)_w(x,0), \\
F^{\sigma,\tau}_{\rho\uparrow,\rR}(x;\zeta)&:=(\cF^{\sigma,\tau}_{\rho\uparrow}e^{(w|\zeta)}I_\rho)_w(0,x). 
\end{align*}
Then $F^{\sigma,\tau}_{\rho\uparrow,\rL}(x;\zeta)$, $F^{\sigma,\tau}_{\rho\uparrow,\rR}(x;\zeta)$ are holomorphically continued to $2D\times\fp^-$, 
and for $f(w)\in\cO(D,V_\rho)_{\widetilde{K}}=\cP(\fp^+)\otimes V_\rho$, we have 
\begin{align*}
&(\cF^{\sigma,\tau}_{\rho\uparrow}f)(x,y)=F^{\sigma,\tau}_{\rho\uparrow,\rC}\biggl(\frac{x-y}{2};\frac{\partial}{\partial w}\biggr)f(w)\biggr|_{w=(x+y)/2} \\
&=F^{\sigma,\tau}_{\rho\uparrow,\rL}\biggl(x-y;\frac{\partial}{\partial w}\biggr)f(w)\biggr|_{w=y}=F^{\sigma,\tau}_{\rho\uparrow,\rR}\biggl(y-x;\frac{\partial}{\partial w}\biggr)f(w)\biggr|_{w=x}. 
\end{align*}
\item Suppose that the map 
\[ \cF_{\rho\downarrow}^{\sigma,\tau}=\cF_{\rho\downarrow,(x,y)}^{\sigma,\tau}\colon\cO(D,V_\sigma)\hotimes\cO(D,V_\tau)\longrightarrow \cO(D,V_\rho)  \]
intertwines the $\widetilde{G}$-action. We define the polynomial $F^{\sigma,\tau}_{\rho\downarrow}(\xi,\eta)\in\cP(\fp^-\oplus\fp^-,\Hom_\BC\allowbreak(V_\sigma\otimes V_\tau,V_\rho))$ by 
\[ F^{\sigma,\tau}_{\rho\downarrow}(\xi,\eta):=\bigl(\cF^{\sigma,\tau}_{\rho\downarrow}e^{(x|\xi)+(y|\eta)}I_{\sigma\otimes\tau}\bigr)_{(x,y)}(0). \]
Then for $f(x,y)\in\cO(D,V_\sigma)\hotimes\cO(D,V_\tau)$, we have 
\[ (\cF^{\sigma,\tau}_{\rho\downarrow}f)(w)=F^{\sigma,\tau}_{\rho\downarrow}\biggl(\frac{\partial}{\partial x},\frac{\partial}{\partial y}\biggr)f(x,y)\biggr|_{x=y=w}. \]
\end{enumerate}
\end{theorem}

\begin{proof}
Since the 1st equality of (1) and (2) are proved in \cite[Theorem 3.10]{N1}, we focus on the 2nd and 3rd equalities of (1). Since $\cF^{\sigma,\tau}_{\rho\uparrow}$ intertwines the $\fp^+$-action, 
for $u\in\fp^+$, $t\in\BR$, if $x,y,x+tu,y+tu\in D$, then we have 
\begin{align*}
&\frac{d}{dt}(\cF^{\sigma,\tau}_{\rho\uparrow}e^{(w|\zeta)}I_\rho)_w(x+tu,y+tu)=\frac{d}{ds}\biggr|_{s=0}(\cF^{\sigma,\tau}_{\rho\uparrow}e^{(w|\zeta)}I_\rho)_w(x+(t+s)u,y+(t+s)u) \\
&=-d(\hat{\sigma}\hotimes\hat{\tau})((u,0,0))(\cF^{\sigma,\tau}_{\rho\uparrow}e^{(w|\zeta)}I_\rho)_w(x+tu,y+tu) \\
&=-(\cF^{\sigma,\tau}_{\rho\uparrow}d\hat{\rho}((u,0,0))e^{(w|\zeta)}I_\rho)_w(x+tu,y+tu) \\
&=\frac{d}{ds}\biggr|_{s=0}(\cF^{\sigma,\tau}_{\rho\uparrow}e^{(w+su|\zeta)}I_\rho)_w(x+tu,y+tu)=(u|\zeta)(\cF^{\sigma,\tau}_{\rho\uparrow}e^{(w|\zeta)}I_\rho)_w(x+tu,y+tu). 
\end{align*}
Hence we have 
\[ (\cF^{\sigma,\tau}_{\rho\uparrow}e^{(w|\zeta)}I_\rho)_w(x+tu,y+tu)=e^{t(u|\zeta)}(\cF^{\sigma,\tau}_{\rho\uparrow}e^{(w|\zeta)}I_\rho)_w(x,y). \]
Since the right hand side is well-defined for $(x,y,t,u)\in D\times D\times\BR\times\fp^+$, 
the left hand side is also holomorphically continued to this region. 
Especially, by putting $tu=-(x+y)/2$, $-y$ or $-x$, we have 
\begin{align*}
&(\cF^{\sigma,\tau}_{\rho\uparrow}e^{(w|\zeta)}I_\rho)_w(x,y) \\
&=e^{(\frac{x+y}{2}|\zeta)}(\cF^{\sigma,\tau}_{\rho\uparrow}e^{(w|\zeta)}I_\rho)_w\biggl(\frac{x-y}{2},\frac{y-x}{2}\biggr)
=e^{(\frac{x+y}{2}|\zeta)}F^{\sigma,\tau}_{\rho\uparrow,\rC}\biggl(\frac{x-y}{2};\zeta\biggr) \\
&=e^{(y|\zeta)}(\cF^{\sigma,\tau}_{\rho\uparrow}e^{(w|\zeta)}I_\rho)_w(x-y,0)=e^{(y|\zeta)}F^{\sigma,\tau}_{\rho\uparrow,\rL}(x-y;\zeta) \\
&=e^{(x|\zeta)}(\cF^{\sigma,\tau}_{\rho\uparrow}e^{(w|\zeta)}I_\rho)_w(0,y-x)=e^{(x|\zeta)}F^{\sigma,\tau}_{\rho\uparrow,\rR}(y-x;\zeta). 
\end{align*}
Then for $f(w)\in\cO(D,V_\rho)_{\widetilde{K}}=\cP(\fp^+)\otimes V_\rho$, since $\exp\bigl(w\big|\frac{\partial}{\partial z}\bigr)f(z)=f(w+z)$ holds, we have 
\begin{align*}
&(\cF^{\sigma,\tau}_{\rho\uparrow}f)(x,y)=\biggl(\cF^{\sigma,\tau}_{\rho\uparrow}\exp\biggl(w\bigg|\frac{\partial}{\partial z}\biggr)I_\rho\biggr)_w(x,y)f(z)\biggr|_{z=0} \\
&=\exp\biggl(y\bigg|\frac{\partial}{\partial z}\biggr)F^{\sigma,\tau}_{\rho\uparrow,\rL}\biggl(x-y;\frac{\partial}{\partial z}\biggr)f(z)\biggr|_{z=0} \\
&=F^{\sigma,\tau}_{\rho\uparrow,\rL}\biggl(x-y;\frac{\partial}{\partial z}\biggr)f(z+y)\biggr|_{z=0}=F^{\sigma,\tau}_{\rho\uparrow,\rL}\biggl(x-y;\frac{\partial}{\partial z}\biggr)f(z)\biggr|_{z=y}. 
\end{align*}
This proves the 2nd equality of (1). The 1st and 3rd equalities are also proved similarly. 
\end{proof}

We note that these theorems hold without the assumption that $G$ is of tube type. 
The symbols of differential symmetry breaking operators are also characterized as solutions of certain differential equations (F-method, see \cite{KP1, KP2}). 
If $\sigma,\tau$ are one-dimensional, then the operator norms of $\tilde{\cF}^{\sigma,\tau}_{\rho,\rK\downarrow}$ are determined in \cite[Corollary 8.7]{N3}. 
When $\fp^+=\BC$, $G=SL(2,\BR)$, the differential symmetry breaking operator $\tilde{\cF}^{\sigma,\tau}_{\rho,\rK\downarrow}$ coincides with the Rankin--Cohen bidifferential operator up to a constant multiple. 

In Theorem \ref{thm_int_expr_D}, the intertwining operators are given as integral operators on the complex domain $D$. 
On the other hand, in \cite{KP3}, the holographic operator for $G=SL(2,\BR)$ is given as an integral operator on some real line segment. 
In this paper, we seek realizations of holographic operators as integral operators on some totally real submanifolds in $D$ or $T_\Omega$.

\section{Definition of contour of integration}\label{section_contour}

In this section, for suitable $(x,y)\in D\times D$ or $T_\Omega\times T_\Omega$, we define a totally real submanifold $C(x,y)\subset D$ or $T_\Omega$ used in the integral operator given later, and check some properties of $C(x,y)$. To do this, first let 
\begin{equation}\label{formula_Dtimes}
\begin{split}
\fp^{+\times}&:=\{x\in\fp^+\mid \det(x)\ne 0\}, \\
D^\times&:=\{x\in D\mid \det(x)\ne 0\}, \\
(D\times D)^\times&:=\{(x,y)\in D\times D\mid \det(x-y)\ne 0\}, \\
(T_\Omega\times T_\Omega)^\times&:=\{(x,y)\in T_\Omega\times T_\Omega\mid \det(x-y)\ne 0\}. 
\end{split}
\end{equation}

\begin{lemma}\label{lem_cov_det}
\begin{enumerate}
\item For $x,y\in T_\Omega$, $g\in {}^c\hspace{-1pt}G$, we have 
\begin{gather*}
P(g.x-g.y)=\kappa(g,x)P(x-y){}^t\hspace{-1pt}\kappa(g,y)=\kappa(g,y)P(x-y){}^t\hspace{-1pt}\kappa(g,x), \\
P(g_c^{-1}.x-g_c^{-1}.y)=\kappa(g_c^{-1},x)P(x-y){}^t\hspace{-1pt}\kappa(g_c^{-1},y)=\kappa(g_c^{-1},y)P(x-y){}^t\hspace{-1pt}\kappa(g_c^{-1},x). 
\end{gather*}
\item For $x,y\in D$, $g\in G$, we have 
\begin{gather*}
P(g.x-g.y)=\kappa(g,x)P(x-y){}^t\hspace{-1pt}\kappa(g,y)=\kappa(g,y)P(x-y){}^t\hspace{-1pt}\kappa(g,x), \\
P(g_c.x-g_c.y)=\kappa(g_c,x)P(x-y){}^t\hspace{-1pt}\kappa(g_c,y)=\kappa(g_c,y)P(x-y){}^t\hspace{-1pt}\kappa(g_c,x). 
\end{gather*}
\item $(D\times D)^\times, (T_\Omega\times T_\Omega)^\times$ are stable under the diagonal actions of $G, {}^c\hspace{-1pt}G$ respectively. 
\end{enumerate}
\end{lemma}

\begin{proof}
(1) For the former equalities, since ${}^c\hspace{-1pt}G$ is generated by $N^+\cup L\cup\{J\}$, it is enough to prove when $g\in N^+, L$, and when $g=J=\exp\bigl(\frac{\pi}{2}(-e,0,e)\bigr)$. 
When $g=\exp((u,0,0))\in N^+$ with $u\in\fn^+$, we have $g.x=x+u$, $\kappa(g,x)=I$, and 
\[ P(g.x-g.y)=P((x+u)-(y+u))=P(x-y). \]
When $g=l\in L$, we have $\kappa(l,x)=l$, and 
\[ P(l.x-l.y)=P(l.(x-y))=lP(x-y){}^t\hspace{-1pt}l. \] 
When $g=J$, we have $J.x=-x^{-1}$, $\kappa(J,x)=P(x)^{-1}$, and by \cite[Lemma X.4.4\,(i)]{FK}, 
\begin{align*}
P(J.x-J.y)=P(-x^{-1}+y^{-1})&=P(x)^{-1}P(x-y)P(y)^{-1}=P(x)^{-1}P(x-y){}^t\hspace{-1pt}P(y)^{-1} \\
&=P(y)^{-1}P(x-y)P(x)^{-1}=P(y)^{-1}P(x-y){}^t\hspace{-1pt}P(x)^{-1}. 
\end{align*}
Hence we get the desired formulas. For the latter formulas, since 
\begin{equation}\label{formula_Cayley_decomp}
g_c^{-1}=\exp\biggl(-\frac{\pi\sqrt{-1}}{4}(e,0,e)\biggr)=\exp((\sqrt{-1}e,0,0))P(\sqrt{-2}e)J\exp((\sqrt{-1}e,0,0)), 
\end{equation}
and since 
\[ J\exp((\sqrt{-1}e,0,0)).x=-(x+\sqrt{-1}e)^{-1} \]
is well-defined for all $x\in T_\Omega$, this is proved similarly to the former formulas. 

(2) This follows from (1) and ${}^c\hspace{-1pt}G=\Int(g_c)G$, $\kappa(g_c,x)=\kappa(g_c^{-1},g_c.x)^{-1}$. 

(3) Since $\det(x)=\chi(P(x))$, this is clear from (1), (2). 
\end{proof}

For later use, we also prove the following. 

\begin{lemma}\label{lem_cov_inv}
\begin{enumerate}
\item For $x,y,w\in\fp^+$, if 
\begin{equation}\label{formula_ass_inv}
\det(x-y),\;\det(w-y),\;\det(x-w)\ne 0, 
\end{equation}
then we have 
\[ \det((w-y)^{-1}+(x-w)^{-1})\ne 0. \]
\item If $x,y,w\in T_\Omega$ satisfy (\ref{formula_ass_inv}), then for $g\in {}^c\hspace{-1pt}G$, we have 
\begin{gather*}
((g.w-g.y)^{-1}+(g.x-g.w)^{-1})^{-1}=\kappa(g,w).((w-y)^{-1}+(x-w)^{-1})^{-1}, \\ 
((g_c^{-1}.w\hspace{-1pt}-\hspace{-1pt}g_c^{-1}.y)^{-1}+(g_c^{-1}.x\hspace{-1pt}-\hspace{-1pt}g_c^{-1}.w)^{-1})^{-1}=\kappa(g_c^{-1},w).((w\hspace{-1pt}-\hspace{-1pt}y)^{-1}+(x\hspace{-1pt}-\hspace{-1pt}w)^{-1})^{-1}.  
\end{gather*}
\item If $x,y,w\in D$ satisfy (\ref{formula_ass_inv}), then for $g\in G$, we have 
\begin{gather*}
((g.w-g.y)^{-1}+(g.x-g.w)^{-1})^{-1}=\kappa(g,w).((w-y)^{-1}+(x-w)^{-1})^{-1}, \\ 
((g_c.w-g_c.y)^{-1}+(g_c.x-g_c.w)^{-1})^{-1}=\kappa(g_c,w).((w-y)^{-1}+(x-w)^{-1})^{-1}.  
\end{gather*}
\end{enumerate}
\end{lemma}

\begin{proof}
(1) By \cite[Lemma X.4.4\,(i)]{FK}, we have 
\begin{align*}
P((w-y)^{-1}+(x-w)^{-1})&=P(w-y)^{-1}P((w-y)+(x-w))P(x-w)^{-1} \\
&=P(w-y)^{-1}P(x-y)P(x-w)^{-1}, 
\end{align*}
and hence we get the desired formula. 

(2) For the former formulas, again it is enough to prove when $g\in N^+, L$, and when $g=J$. When $g=\exp((u,0,0))\in N^+$ with $u\in\fn^+$, we have $g.x=x+u$, $\kappa(g,x)=I$, and 
\begin{align*}
((g.w-g.y)^{-1}+(g.x-g.w)^{-1})^{-1}&=(((w+u)-(y+u))^{-1}+((x+u)-(w+u))^{-1})^{-1} \\&=((w-y)^{-1}+(x-w)^{-1})^{-1}. 
\end{align*}
When $g=l\in L$, we have $\kappa(l,x)=l$, and 
\begin{align*}
((l.w-l.y)^{-1}+(l.x-l.w)^{-1})^{-1}&=({}^t\hspace{-1pt}l^{-1}.(w-y)^{-1}+{}^t\hspace{-1pt}l^{-1}.(x-w)^{-1})^{-1} \\&=l.((w-y)^{-1}+(x-w)^{-1})^{-1}. 
\end{align*}
When $g=J$, we have $J.x=-x^{-1}$, $\kappa(J,x)=P(x)^{-1}$. 
By putting $P(w^{-1/2})x=:x'$, $P(w^{-1/2})y=:y'$, we have 
\begin{align*}
&((J.w-J.y)^{-1}+(J.x-J.w)^{-1})^{-1}=((-w^{-1}+y^{-1})^{-1}+(-x^{-1}+w^{-1}))^{-1} \\
&=\bigl((P(w^{-1/2})(-e+y^{\prime-1}))^{-1}+(P(w^{-1/2})(-x^{\prime-1}+e))^{-1}\bigr)^{-1} \\
&=P(w^{-1/2})\bigl((-e+y^{\prime-1})^{-1}+(-x^{\prime-1}+e)^{-1}\bigr)^{-1} \\
&=P(w^{-1/2})\bigl((e-y')^{-1}-e+(x'-e)^{-1}+e\bigr)^{-1} \\
&=P(w^{-1/2})\bigl((P(w^{-1/2})(w-y))^{-1}+(P(w^{-1/2})(x-w))^{-1}\bigr)^{-1} \\
&=P(w^{-1/2})^2\bigl((w-y)^{-1}+(x-w)^{-1}\bigr)^{-1}=P(w)^{-1}\bigl((w-y)^{-1}+(x-w)^{-1}\bigr)^{-1}. 
\end{align*}
The latter formulas follow from (\ref{formula_Cayley_decomp}) and an argument similar to the above. 

(3) Clear from (2) and ${}^c\hspace{-1pt}G=\Int(g_c)G$. 
\end{proof}

Next, for $(x,y)\in (D\times D)^\times$, we take $g\in G$ such that $g^{-1}.y=0$, $g^{-1}.x\in\Omega$, 
and define the totally real submanifold $C(x,y)\subset D$ by 
\begin{subequations}\label{formula_contour}
\begin{equation}
C(x,y)=C(x,y;g):=g.(\Omega\cap (g^{-1}.x-\Omega)). 
\end{equation}
Similarly, for $(x,y)\in (T_\Omega\times T_\Omega)^\times$, we define $C(x,y)\subset T_\Omega$ by 
\begin{equation}
C(x,y):=g_c.C(g_c^{-1}.x,g_c^{-1}.y). 
\end{equation}
\end{subequations}
When $\fp^+=\BC$, these coincide with the geodesic intervals connecting $x$ and $y$ with respect to the Poincar\'e metrics on $D$ and $T_\Omega$. 

\begin{proposition}\label{prop_welldef_contour}
For $(x,y)\in (D\times D)^\times$, $C(x,y;g)$ does not depend on the choice of $g\in G$ satisfying $g^{-1}.y=0$, $g^{-1}.x\in\Omega$. 
\end{proposition}

To prove this, we need the following lemma. 
\begin{lemma}\label{lem_KL}
Let $k\in K$, $x\in\Omega$. If $k.x\in\Omega$, then $k\in K_L$. 
\end{lemma}

\begin{proof}
Let $y:=k.x\in\Omega$. Then we have 
\[ kP(x)=kP(x){}^t\hspace{-1pt}k\overline{k}=P(k.x)\overline{k}=P(y)\overline{k}\in\End_\BC(\fp^+) \]
(under the identification of $k\in K$ and $\Ad(k)|_{\fp^+}\in GL_\BC(\fp^+)$). 
Since $P(x), P(y)\in\End_\BR(\fn^+)\subset\End_\BC(\fp^+)$ are positive definite, these are diagonalizable with positive real eigenvalues. If $v, w\in\fn^+$ are eigenvectors of $P(x), P(y)$ with the eigenvalues $\xi,\eta>0$ respectively, then we have 
\[ \xi(kv|w)=(kP(x)v|w)=(P(y)\overline{k}v|w)=(\overline{k}v|P(y)w)=\eta(\overline{k}v|w). \]
Hence $\xi=\eta$ implies $(kv|w)=(\overline{k}v|w)$. Similarly, since $|(kv|w)|=|(\overline{k}v|w)|$ holds, 
$\xi\ne\eta$ also implies $(kv|w)=(\overline{k}v|w)=0$. Since this holds for all eigenvectors in the real form $\fn^+\subset\fp^+$, 
we get $k=\overline{k}\in K^\BC\cap\End_\BR(\fn^+)=L\cup(-L)$ by \cite[Proposition VIII.2.8]{FK}, 
where an element of $-L$ sends $\Omega$ to $-\Omega$. Since $y=k.x\in\Omega$, we have $k\in L\cap K=K_L$. 
\end{proof}

\begin{proof}[Proof of Proposition \ref{prop_welldef_contour}]
Let $(x,y)\in (D\times D)^\times$. Suppose $g_1,g_2\in G$, $u_1,u_2\in \Omega$ satisfy 
\[ g_1^{-1}.y=g_2^{-1}.y=0, \qquad g_1^{-1}.x=u_1, \qquad g_2^{-1}.x=u_2. \]
Let $k:=g_2^{-1}g_1$. Then since $k.0=0$, we have $k\in\Stab_G(0)=K$, and since $k.u_1=u_2$, by Lemma \ref{lem_KL}, we have $k\in K_L$. 
Since $k$ stabilizes $\Omega$, we have 
\begin{align*}
C(x,y;g_1)&=g_1.(\Omega\cap (u_1-\Omega))=g_2k.(\Omega\cap (u_1-\Omega)) \\
&=g_2.(\Omega\cap (k.u_1-\Omega))=g_2.(\Omega\cap (u_2-\Omega))=C(x,y;g_2). \qedhere 
\end{align*}
\end{proof}

By the definition of $C(x,y)$, for $g\in G$, $(x,y)\in (D\times D)^\times$, we have 
\[ C(g.x,g.y)=g.C(x,y). \]

Next, for $x,y\in D$, let $\dist(x,y)$ be the distance with respect to the $G$-invariant Hermitian metric 
\[ g_x(u,\overline{v}):=D_u\overline{D_v}\log h(x,\overline{x})^{-1}=(B(x,\overline{x})^{-1}u|\overline{v}) \qquad (x\in D,\;u,\overline{v}\in T_xD=\fp^+), \]
where $D_uf(x):=\frac{1}{2}\frac{d}{dt}\bigl|_{t=0}(f(x+tu)-\sqrt{-1}f(x+t\sqrt{-1}u))$ and 
$\overline{D_v}f(x):=\frac{1}{2}\frac{d}{dt}\bigl|_{t=0}(f(x+t\overline{v})+\sqrt{-1}f(x+t\sqrt{-1}\overline{v}))$ are the holomorphic and anti-holomorphic directional derivatives. 
Especially, for $u\in \Omega$, the curve $t\mapsto\exp(t(u,0,u)).0=\tanh(tu)$ is a geodesic connecting $0$ and $\tanh(u)$, and we have 
\[ \dist(\tanh(u),0)^2=(u|u)=\tr(u^2). \]
Also, for $x\in D$, $r\ge 0$, let $B_x(r):=\{y\in D\mid \dist(x,y)\le r\}\subset D$ be the closed ball with the center $x$, radius $r$. 
Then $C(x,y)$ is bounded as follows. 

\begin{lemma}\label{lem_bdd}
For $(x,y)\in (D\times D)^\times$, we have $C(x,y)\subset B_y(\dist(x,y))$. 
\end{lemma}

\begin{proof}
By the $G$-invariance of $\dist$, it is enough to show, for $u\in \Omega\cap D$, 
\[ C(u,0)=\Omega\cap(u-\Omega)\subset B_0(\dist(u,0)). \]
Indeed, for $z\in \Omega\cap D$, we have 
\[ \dist(z,0)^2=\tr((\artanh(z))^2)=\tr\biggl(\biggl(\sum_{j=0}^\infty \frac{z^{2j+1}}{2j+1}\biggr)^2\biggr)=\sum_{j,k=0}^\infty \frac{\tr(z^{2j+2k+2})}{(2j+1)(2k+1)}. \]
Suppose $z\in C(u,0)=\Omega\cap(u-\Omega)$. Then since $(u|x)>(z|x)$ holds for any $x\in \Omega$, by putting $x=u^j\circ z^k$, we have 
\[ (u^{j+1}|z^k)=(u\circ u^j|z^k)=(u|u^j\circ z^k)>(z|u^j\circ z^k)=(u^j|z\circ z^k)=(u^j|z^{k+1}), \]
and inductively, we have 
\[ \tr(u^j)>(u^{j-1}|z)>(u^{j-2}|z^2)>\cdots>(u|z^{j-1})>\tr(z^j). \]
Hence we get 
\[ \dist(u,0)^2=\sum_{j,k=0}^\infty \frac{\tr(u^{2j+2k+2})}{(2j+1)(2k+1)}>\sum_{j,k=0}^\infty \frac{\tr(z^{2j+2k+2})}{(2j+1)(2k+1)}=\dist(z,0)^2, \]
and get the desired inclusion. 
\end{proof}

We end this section by proving the following. 

\begin{lemma}\label{lem_ori_contour}
\begin{enumerate}
\item If $x,y\in D\cap\fn^+$ and $x-y\in\Omega$, then we have 
\[ C(x,y)=(y+\Omega)\cap (x-\Omega). \]
\item For $x,y\in (D\times D)^\times$, we have 
\[ C(x,y)=C(y,x), \]
with the same orientation if $n=\dim\fp^+$ is even, and with the reversed orientation if $n=\dim\fp^+$ is odd. 
\end{enumerate}
\end{lemma}

\begin{proof}
(1) For $y\in D\cap\fn^+$, $t\in[0,1]$, let $v:=\artanh(y)\in\fn^+$, $g_t:=\exp(t(v,0,v))\in G$, so that $g_1.0=\tanh(v)=y$. 
Then $g_t$ preserves $D\cap\fn^+$. By Lemma \ref{lem_cov_det}, 
\[ \det(g_t^{-1}.x-g_t^{-1}.y)=\chi(\kappa(g_t^{-1},x))\chi(\kappa(g_t^{-1},y))\det(x-y)>0, \]
and hence $x-y$ and $g_t^{-1}.x-g_t^{-1}.y$ belong to the same connected component of $\{z\in\fn^+\mid \det(z)\ne 0\}$ for all $t\in[0,1]$. Since $x-y\in\Omega$ by assumption, we have $g_t^{-1}.x-g_t^{-1}.y\in\Omega$, and especially, $u:=g_1^{-1}.x\in\Omega$. 
Therefore, by the definition of $C(x,y)$, we have 
\[ C(x,y)=C(g_1.u,g_1.0)=g_1.(\Omega\cap(u-\Omega)). \]
Now since $\Omega\cap(u-\Omega)$ coincides with the unique bounded connected component of 
\[ \{z\in\fn^+\mid \det(z)\det(u-z)\ne 0\}, \]
$C(g_t.u,g_t.0)=g_t.(\Omega\cap(u-\Omega))$ coincides with the unique bounded connected component of 
\[ \{w\in\fn^+\mid \det(w-g_t.0)\det(g_t.u-w)\ne 0\}, \]
which is equal to $(g_t.0+\Omega)\cap(g_t.u-\Omega)$. By putting $t=1$, we get the desired formula. 

(2) It is enough to show $C(u,0)=C(0,u)$ for $u\in D\cap\Omega$ with the desired orientation. Indeed, by (1), we have 
\[ C(0,u)=-C(0,-u)=-((-u+\Omega)\cap(-\Omega))=\Omega\cap(u-\Omega)=C(u,0). \]
Let $v:=\artanh(u)\in\fn^+$, $g_1:=\exp((v,0,v))\in G$ so that $g_1.0=u$, 
and let $i:=P(\sqrt{-1}e)\in K$ so that $i.z=-z$ for all $z\in\fp^+$. 
Then since $g_1^{-1}.0=-u$ holds, $z\mapsto ig_1^{-1}.z=-(g_1^{-1}.z)$ maps $C(u,0)$ to $C(0,u)$, 
and the holomorphic top form $dz=dz_1\wedge\cdots\wedge dz_n$ is transformed as 
\[ d(-(g_1^{-1}.z))=\chi(\kappa(ig_1^{-1},z))^{\frac{2n}{r}}dz=\chi(i)^{\frac{2n}{r}}\chi(\kappa(g_1^{-1},z))^{\frac{2n}{r}}dz=(-1)^n\chi(\kappa(g_1^{-1},z))^\frac{2n}{r}dz. \]
Since $\chi(\kappa(g,z))>0$ for all $g\in G\cap{}^c\hspace{-1pt}G$, $z\in\fn^+$, we get the desired orientation. 
\end{proof}

\section{Construction of integral holographic operators}\label{section_construction}
\subsection{Main theorems}

We give the main theorems of this paper. Let $(D\times D)^\times, (T_\Omega\times T_\Omega)^\times$ be as in (\ref{formula_Dtimes}), 
and for $(x,y)\in(D\times D)^\times$ or $(T_\Omega\times T_\Omega)^\times$, let $C(x,y)$ be as in (\ref{formula_contour}). 

\begin{theorem}\label{main_thm}
Let $(\sigma,V_\sigma), (\tau,V_\tau)$ be irreducible representations of $\widetilde{K}$ with the restricted lowest weights 
$-\frac{1}{2}\sum_{j=1}^r\lambda_j\gamma_j$, $-\frac{1}{2}\sum_{j=1}^r\mu_j\gamma_j\in(\fa_\fl^\BC)^\vee$ respectively. 
We choose $m\in\BZ$ such that 
$\Re\lambda_r+m>\frac{n}{r}-1$, $\Re\mu_r+m>\frac{n}{r}-1$ hold. 
Let $(\rho,V_\rho)$ be an irreducible representation of $\widetilde{K}$ such that 
a non-zero operator-valued polynomial 
\[ \rK(x)\in\bigl(\cP(\fp^+)\det(x)^m\otimes\Hom_\BC(V_\rho,V_\sigma\otimes V_\tau)\bigr)^{\widetilde{K}} \]
exists. Then the maps 
\begin{align*}
&\cF^{\sigma,\tau}_{\rho,\rK\uparrow}\colon\cO(D,V_\rho)\longrightarrow
\begin{cases} \cO(D\times D,V_\sigma\otimes V_\tau) & (\text{if }2m\ge \lambda_1+\mu_1-\lambda_r-\mu_r), \\
\cO((D\times D)^\times,V_\sigma\otimes V_\tau) & (\text{if }2m< \lambda_1+\mu_1-\lambda_r-\mu_r),\end{cases} \\
&\cF^{\sigma,\tau}_{\rho,\rK\uparrow}\colon\cO(T_\Omega,V_\rho)\longrightarrow
\begin{cases} \cO(T_\Omega\times T_\Omega,V_\sigma\otimes V_\tau) & (\text{if }2m\ge \lambda_1+\mu_1-\lambda_r-\mu_r), \\
\cO((T_\Omega\times T_\Omega)^\times,V_\sigma\otimes V_\tau) & (\text{if }2m< \lambda_1+\mu_1-\lambda_r-\mu_r),\end{cases}
\end{align*}
given by the same formula 
\begin{multline}
(\cF^{\sigma,\tau}_{\rho,\rK\uparrow} f)(x,y):=\det(x-y)^{\frac{n}{r}}(\sigma\otimes\tau)(P(x-y))\int_{C(x,y)} \det(w-y)^{-\frac{n}{r}}\det(x-w)^{-\frac{n}{r}} \\
{}\times \bigl(\sigma(P(w-y)^{-1})\otimes \tau(P(x-w)^{-1})\bigr)\rK\bigl(((w-y)^{-1}+(x-w)^{-1})^{-1}\bigr) f(w)\,dw \label{formula_mainthm}
\end{multline}
intertwines the $\widetilde{G}$-action. 
\end{theorem}

The result for $T_\Omega$ follows from that for $D$ and 
\[ \cF^{\sigma,\tau}_{\rho,\rK\uparrow,T_\Omega}=(\hat{\sigma}\hotimes\hat{\tau})(g_c)\circ\cF^{\sigma,\tau}_{\rho,\rK\uparrow,D}\circ\hat{\rho}(g_c^{-1}), \]
which is proved as in the proof of Lemma \ref{lem_intertwining} below. 
In the following, we give the results and proofs only for $D$. 

We note that $C(x,y)\subset D$ is defined only for $(x,y)\in(D\times D)^\times$, namely, for $\det(x-y)\ne 0$ case, and hence $\cF^{\sigma,\tau}_{\rho,\rK\uparrow}f$ is originally defined as a function on $(D\times D)^\times$. 
If $2m\ge \lambda_1+\mu_1-\lambda_r-\mu_r$, then this is holomorphically continued to the function on $D\times D$. 
On the other hand, without this assumption, the holomorphy at $\{\det(x-y)=0\}$ does not hold in general, even if $m\in\BZ_{\ge 0}$. 
However, if $(\rho,V_\rho)$ satisfies some additional assumption, then the image becomes holomorphic on $D\times D$. 

\begin{theorem}\label{main_thm2}
Let $(\sigma,V_\sigma)$, $(\tau,V_\tau)$, $(\rho,V_\rho)$ and $m\in\BZ$ be as in Theorem \ref{main_thm}. We assume $m\in\BZ_{\ge 0}$, $(\hat{\rho},\cO(D,V_\rho))$ is irreducible, and additionally assume 
\[ \bigl(\cP(\fp^+)[\det^{-1}]\otimes\Hom_\BC(V_\rho,V_\sigma\otimes V_\tau)\bigr)^{\widetilde{K}}
=\bigl(\cP(\fp^+)\otimes\Hom_\BC(V_\rho,V_\sigma\otimes V_\tau)\bigr)^{\widetilde{K}}. \]
Then the map 
\[ \cF^{\sigma,\tau}_{\rho,\rK\uparrow}\colon\cO(D,V_\rho)\longrightarrow
\cO(D\times D,V_\sigma\otimes V_\tau)  \]
defined by (\ref{formula_mainthm}) intertwines the $\widetilde{G}$-action. 
\end{theorem}

Especially, if $\sigma=\chi^{-\lambda}$, $\tau=\chi^{-\mu}$ with $\lambda,\mu\in\BC$, 
then $\lambda_1=\lambda_r=\lambda$, $\mu_1=\mu_r=\mu$ hold, 
and by (\ref{formula_HKS}), $\rho$ satisfies the assumption if and only if $\rho=\chi^{-\lambda-\mu}\otimes V_{\bk\gamma}^\vee$ for some $\bk\in\BZ_{++}^r$ with $k_r\ge m$, 
where $V_{\bk\gamma}^\vee$ is the irreducible representation of $K$ with the lowest weight $-(k_1\gamma_1+\cdots+k\gamma_r)\in(\ft^\BC)^\vee$. 
Therefore, the following holds. 

\begin{corollary}\label{cor_scalar1}
Let $\lambda,\mu\in\BC$, $\bk\in\BZ_{++}^r$, $\Re\lambda, \Re\mu>-k_r+\frac{n}{r}-1$. We fix a non-zero polynomial 
\[ \rK_\bk(x)\in\bigl(\cP(\fp^+)\otimes \Hom_\BC(V_{\bk\gamma}^\vee,\BC)\bigr)^K, \]
which is unique up to a constant multiple. Then the map 
\begin{gather*}
\cF^{\lambda,\mu}_{\bk\uparrow}\colon\cO_{\lambda+\mu}(D,V_{\bk\gamma}^\vee)\longrightarrow\cO_\lambda(D)\hotimes\cO_\mu(D), \\
\begin{split}
(\cF^{\lambda,\mu}_{\bk\uparrow} f)(x,y):=\det(x-y)^{-\lambda-\mu+\frac{n}{r}}\int_{C(x,y)} \det(w-y)^{\lambda-\frac{n}{r}}\det(x-w)^{\mu-\frac{n}{r}} \\
{}\times \rK_\bk\bigl(((w-y)^{-1}+(x-w)^{-1})^{-1}\bigr) f(w)\,dw 
\end{split}
\end{gather*}
intertwines the $\widetilde{G}$-action. 
\end{corollary}

When $\bk=(l,\ldots,l)$, we have $\cO_\lambda(D,V_{\bk\gamma}^\vee)\simeq\cO_{\lambda+2l}(D)$, and may assume $\rK_\bk(x)=\det(x)^l$. 
In this case, by \cite[Lemma X.4.4\,(i)]{FK} we have 
\begin{equation}\label{formula_det_inv}
\det\bigl(((w-y)^{-1}+(x-w)^{-1})^{-1}\bigr)=\det(w-y)\det(x-y)^{-1}\det(x-w), 
\end{equation}
and the above corollary is simplified as follows. 

\begin{corollary}\label{cor_scalar2}
Let $\lambda,\mu\in\BC$, $l\in\BZ_{\ge 0}$, $\Re\lambda, \Re\mu>-l+\frac{n}{r}-1$. Then the map 
\begin{gather*}
\cF^{\lambda,\mu}_{l\uparrow}\colon\cO_{\lambda+\mu+2l}(D)\longrightarrow\cO_\lambda(D)\hotimes\cO_\mu(D), \\
(\cF^{\lambda,\mu}_{l\uparrow} f)(x,y):=\det(x-y)^{-\lambda-\mu-l+\frac{n}{r}}\int_{C(x,y)} f(w)\det(w-y)^{\lambda+l-\frac{n}{r}}\det(x-w)^{\mu+l-\frac{n}{r}}\,dw 
\end{gather*}
intertwines the $\widetilde{G}$-action. 
\end{corollary}

\begin{example}
Let $\fp^+=\Sym(r,\BC)$, $G=Sp(r,\BR)$. Let $\ft\subset\fk$ be a Cartan subalgebra, and we identify $(\ft^\BC)^\vee\simeq\BC^r$ with the standard basis $\{\be_1,\ldots,\be_r\}\in\BC^r$, such that $\gamma_j=2\be_j$ holds. 

Let $V_\sigma:=\chi^{-\lambda}\otimes(\BC^r)^\vee$, $V_\tau:=\chi^{-\mu}\otimes(\BC^r)^\vee$. 
These have the lowest weights $-((\lambda+1)\be_1+\lambda(\be_2+\cdots+\be_r))$ and $-((\mu+1)\be_1+\mu(\be_2+\cdots+\be_r))$ respectively, 
and hence $(\lambda_1+\mu_1-\lambda_r-\mu_r)/2=1$. Then we have 
\begin{align*}
&\cP(\fp^+)\otimes V_\sigma\otimes V_\tau
\simeq\chi^{-\lambda-\mu}\otimes\bigoplus_{\bk\in\BZ_{++}^r}V_{2\bk}^\vee\otimes V_{\be_1}^\vee\otimes V_{\be_1}^\vee \\
&\simeq\chi^{-\lambda-\mu}\otimes\bigoplus_{\bk\in\BZ_{++}^r}\bigoplus_{\substack{1\le i,j\le r \\ 2\bk+\be_i+\be_j\in\BZ_{++}^r}}V_{2\bk+\be_i+\be_j}^\vee \\
&\simeq\chi^{-\lambda-\mu}\otimes\biggl(\bigoplus_{\bk\in\BZ_{++}^r}\bigoplus_{\substack{1\le i<j\le r \\ 2\bk+\be_i+\be_j\in\BZ_{++}^r}}(V_{2\bk+\be_i+\be_j}^\vee)^{\oplus 2}\oplus\bigoplus_{\substack{\bk\in\BZ_{++}^r \\ k_1\ge 1}}(V_{2\bk}^\vee)^{\oplus\sharp\{1\le j\le r\mid k_j>k_{j+1}\}}\biggr),
\end{align*}
where $k_{r+1}:=0$, and hence, for $\lambda,\mu>\frac{2n}{r}-1=r$, we have 
\begin{align*}
&\cH_\lambda(D,(\BC^r)^\vee)\hotimes\cH_\mu(D,(\BC^r)^\vee) \\
&\simeq\bigoplus_{\bk\in\BZ_{++}^r}\bigoplus_{\substack{1\le i<j\le r \\ 2\bk+\be_i+\be_j\in\BZ_{++}^r}}\cH_{\lambda+\mu}(D,V_{2\bk+\be_i+\be_j}^\vee)^{\oplus 2}\oplus\bigoplus_{\substack{\bk\in\BZ_{++}^r \\ k_1\ge 1}}\cH_{\lambda+\mu}(D,V_{2\bk}^\vee)^{\oplus\sharp\{1\le j\le r\mid k_j>k_{j+1}\}}, 
\end{align*}
where $k_{r+1}:=0$. Similarly, we have 
\begin{align*}
&\cP(\fp^+)[\det^{-1}]\otimes V_\sigma\otimes V_\tau
\simeq\chi^{-\lambda-\mu}\otimes\bigoplus_{\bk\in\BZ_{+}^r}V_{2\bk}^\vee\otimes V_{\be_1}^\vee\otimes V_{\be_1}^\vee \\
&\simeq\chi^{-\lambda-\mu}\otimes\bigoplus_{\bk\in\BZ_{+}^r}\bigoplus_{\substack{1\le i,j\le r \\ 2\bk+\be_i+\be_j\in\BZ_{+}^r}}V_{2\bk+\be_i+\be_j}^\vee \\
&\simeq\chi^{-\lambda-\mu}\otimes\biggl(\bigoplus_{\bk\in\BZ_{+}^r}\bigoplus_{\substack{1\le i<j\le r \\ 2\bk+\be_i+\be_j\in\BZ_{+}^r}}(V_{2\bk+\be_i+\be_j}^\vee)^{\oplus 2}\oplus\bigoplus_{\bk\in\BZ_{+}^r}(V_{2\bk}^\vee)^{\oplus\sharp\{1\le j\le r\mid k_j>k_{j+1}\}}\biggr),
\end{align*}
where $k_{r+1}:=-\infty$. 
\begin{enumerate}
\item Suppose $V_\rho=\chi^{-\lambda-\mu}\otimes V_{2\bk+\be_i+\be_j}^\vee$ with $\bk\in\BZ_{++}^r$, $1\le i<j\le r$, $2\bk+\be_i+\be_j\in\BZ_{++}^r$. Then we have 
\[ \Hom_{\widetilde{K}}(V_\rho,\cP(\fp^+)\otimes V_\sigma\otimes V_\tau)=\Hom_{\widetilde{K}}(V_\rho,\cP(\fp^+)[\det^{-1}]\otimes V_\sigma\otimes V_\tau). \]
Hence, if one of the conditions 
\begin{align*}
&\textstyle k_r\ge 1,\; \Re\lambda+k_r>\frac{r-1}{2},\; \Re\mu+k_r>\frac{r-1}{2},\\
\text{or}\; &\textstyle k_r=0,\; \Re\lambda>\frac{r-1}{2},\; \Re\mu>\frac{r-1}{2},\; \cO(D,V_\rho)\colon\text{irreducible},
\end{align*}
is satisfied, then 
\[ \cF^{\sigma,\tau}_{\rho,\rK\uparrow}\colon\cO(D,V_\rho)\longrightarrow
\cO(D\times D,V_\sigma\otimes V_\tau) \]
defined by (\ref{formula_mainthm}) intertwines the $\widetilde{G}$-action. 
\item Suppose $V_\rho=\chi^{-\lambda-\mu}\otimes V_{2\bk}^\vee$ with $\bk\in\BZ_{++}^r$, $k_r\ge 1$. Then we have 
\[ \Hom_{\widetilde{K}}(V_\rho,\cP(\fp^+)\otimes V_\sigma\otimes V_\tau)=\Hom_{\widetilde{K}}(V_\rho,\cP(\fp^+)[\det^{-1}]\otimes V_\sigma\otimes V_\tau). \]
Hence, if one of the conditions 
\begin{align*}
&\begin{cases} k_r\ge 1,\; \rK(x)\in\bigl(\cP(\fp^+)\det(x)^{k_r}\otimes\Hom_\BC(V_\rho,V_\sigma\otimes V_\tau)\bigr)^{\widetilde{K}},\\ \Re\lambda+k_r>\frac{r-1}{2},\; \Re\mu+k_r>\frac{r-1}{2},\end{cases} \\
&\begin{cases} k_r\ge 2,\; \rK(x)\in\bigl(\cP(\fp^+)\det(x)^{k_r-1}\otimes\Hom_\BC(V_\rho,V_\sigma\otimes V_\tau)\bigr)^{\widetilde{K}},\\ \Re\lambda+k_r-1>\frac{r-1}{2},\; \Re\mu+k_r-1>\frac{r-1}{2},\end{cases} 
\\\text{or }
&\begin{cases} k_r=1,\; \rK(x)\in\bigl(\cP(\fp^+)\otimes\Hom_\BC(V_\rho,V_\sigma\otimes V_\tau)\bigr)^{\widetilde{K}},\\ \Re\lambda>\frac{r-1}{2},\; \Re\mu>\frac{r-1}{2},\; \cO(D,V_\rho)\colon\text{irreducible}\end{cases} 
\end{align*}
is satisfied, then 
\[ \cF^{\sigma,\tau}_{\rho,\rK\uparrow}\colon\cO(D,V_\rho)\longrightarrow
\cO(D\times D,V_\sigma\otimes V_\tau) \]
defined by (\ref{formula_mainthm}) intertwines the $\widetilde{G}$-action. 
\item Suppose $V_\rho=\chi^{-\lambda-\mu}\otimes V_{2\bk}^\vee$ with $\bk\in\BZ_{++}^r$, $k_1\ge 1$, $k_r=0$. Then we have 
\[ \Hom_{\widetilde{K}}(V_\rho,\cP(\fp^+)\otimes V_\sigma\otimes V_\tau)\subsetneq\Hom_{\widetilde{K}}(V_\rho,\cP(\fp^+)[\det^{-1}]\otimes V_\sigma\otimes V_\tau). \]
Hence, if $\Re\lambda>\frac{r-1}{2}$ and $\Re\mu>\frac{r-1}{2}$ hold, then 
\[ \cF^{\sigma,\tau}_{\rho,\rK\uparrow}\colon\cO(D,V_\rho)\longrightarrow
\cO((D\times D)^\times,V_\sigma\otimes V_\tau) \]
defined by (\ref{formula_mainthm}) intertwines the $\widetilde{G}$-action, but the image is not contained in $\cO(D\times D,V_\sigma\otimes V_\tau)$ in general. 
\end{enumerate}
\end{example}

\begin{remark}
The holographic operators $\cF^{\sigma,\tau}_{\rho,\rK\uparrow}$ in Theorem \ref{main_thm} and $\tilde{\cF}^{\sigma,\tau}_{\rho,\rK\uparrow}$ in Theorem \ref{thm_int_expr_D} are in general not proportional even if these are constructed from a common $\rK(x)$, unless the multiplicities of $\cO(D,V_\rho)$ in $\cO(D\times D, V_\sigma \otimes V_\tau)$ and $\cO((D\times D)^\times, V_\sigma \otimes V_\tau)$ are both one. 
\end{remark}

In Sections \ref{subsection_singleval}--\ref{subsection_hol}, we give a proof of Theorem \ref{main_thm}. 
We extend $\hat{\sigma}$ to a representation on $L^\infty_{\mathrm{loc}}(D,V_\sigma)$ by the same formula. 
It is enough to verify 
\begin{enumerate}
\item The integrand of $\cF^{\sigma,\tau}_{\rho,\rK\uparrow}$ is single-valued, 
\item $\cF^{\sigma,\tau}_{\rho,\rK\uparrow}\colon L^\infty_{\mathrm{loc}}(D,V_\rho)\to L^\infty_{\mathrm{loc}}(D\times D,V_\sigma\otimes V_\tau)\det(x-y)^{m-\lceil(\lambda_1+\mu_1-\lambda_r-\mu_r)/2\rceil}$ is continuous, 
\item $(\hat{\sigma}\hotimes\hat{\tau})(g)\circ\cF^{\sigma,\tau}_{\rho,\rK\uparrow}=\cF^{\sigma,\tau}_{\rho,\rK\uparrow}\circ\hat{\rho}(g)$ on $L^\infty_{\mathrm{loc}}$, 
\item $\cF^{\sigma,\tau}_{\rho,\rK\uparrow}(\cO(D,V_\rho))\subset\cO(D\times D,V_\sigma\otimes V_\tau)\det(x-y)^{m-\lceil(\lambda_1+\mu_1-\lambda_r-\mu_r)/2\rceil}$. 
\end{enumerate}
After that, we give a proof of Theorem \ref{main_thm2} in Section \ref{subsection_thm2}.

\subsection{Proof of single-valuedness of the integrand}\label{subsection_singleval}

In this subsection, we prove that the integrand of $\cF^{\sigma,\tau}_{\rho,\rK\uparrow}$ is single-valued. 

\begin{lemma}
Let $(\sigma,V_\sigma),(\tau,V_\tau)$ be irreducible representations of $\widetilde{K}$, and let 
\[ X:=\{(x,w,y)\in D\times D\times D\mid \det(x-y)\ne0,\; w\in C(x,y)\}. \]
Then the functions 
\begin{align*}
X&\longrightarrow GL_\BC(V_\sigma), & (x,w,y)&\longmapsto \sigma(P(x-y)P(w-y)^{-1}),\\
X&\longrightarrow GL_\BC(V_\tau), & (x,w,y)&\longmapsto \tau(P(x-y)P(x-w)^{-1}),\\
X&\longrightarrow \BC^\times, & (x,w,y)&\longmapsto \det(x-y)^{\frac{n}{r}}\det(w-y)^{-\frac{n}{r}}\det(x-w)^{-\frac{n}{r}}
\end{align*}
are single-valued. 
\end{lemma}

\begin{proof}
First, let 
\begin{align*}
X_1:\hspace{-3pt}&=\{(x,w)\in D\times D\mid \det(x)\ne0,\; w\in C(x,0)\} \\
&=\{(k.u,k.z)\in D\times D\mid k\in K,\; u,z,u-z\in\Omega\}, \\
X_2:\hspace{-3pt}&=X_1\cap(\Omega\times\Omega)=\{(u,z)\in \Omega\times\Omega\mid u-z\in\Omega\}. 
\end{align*}
Then since $\exp(\fp)\simeq G/K\simeq D$, we have the diffeomorphism 
\begin{gather*}
\exp(\fp)\times X_1\longrightarrow X, \qquad
(p,x,w)\longmapsto (p.x,p.w,p.0). 
\end{gather*}
Hence it suffices to show that the functions 
\begin{align}
&\exp(\fp)\times X_1\longrightarrow GL_\BC(V_\sigma), \qquad (p,x,w)\longmapsto \sigma(P(p.x-p.0)P(p.w-p.0)^{-1}),\notag\\
&\exp(\fp)\times X_1\longrightarrow GL_\BC(V_\tau), \qquad (p,x,w)\longmapsto \tau(P(p.x-p.y)P(p.x-p.w)^{-1}),\label{formula_singleval1}\\
&\exp(\fp)\times X_1\longrightarrow \BC^\times, \quad (p,x,w)\longmapsto \det(p.x-p.0)^{\frac{n}{r}}\det(p.w-p.0)^{-\frac{n}{r}}\det(p.x-p.w)^{-\frac{n}{r}} \notag
\end{align}
are single-valued. Indeed, for $p\in\exp(\fp)$ and for $(x,w)\in\widetilde{X}_1$ in the universal covering space of $X_1$, by Lemma \ref{lem_cov_det}, we have 
{\setlength{\belowdisplayskip}{0pt}
\begin{align*}
\sigma(P(p.x-p.0)P(p.w-p.0)^{-1})&=\sigma\bigl(\kappa(p,x)P(x-0){}^t\hspace{-1pt}\kappa(p,0)(\kappa(p,w)P(w-0){}^t\hspace{-1pt}\kappa(p,0))^{-1}\bigr) \\*
&=\sigma(\kappa(p,x))\sigma(P(x)P(w)^{-1})\sigma(\kappa(p,w)^{-1}), \\
\tau(P(p.x-p.0)P(p.x-p.w)^{-1})&=\tau\bigl(\kappa(p,0)P(x-0){}^t\hspace{-1pt}\kappa(p,x)(\kappa(p,w)P(x-w){}^t\hspace{-1pt}\kappa(p,x))^{-1}\bigr) \\*
&=\tau(\kappa(p,0))\tau(P(x)P(x-w)^{-1})\tau(\kappa(p,w)^{-1}), 
\end{align*}}
{\setlength{\abovedisplayskip}{0pt}
\begin{align}
&\det(p.x-p.0)^{\frac{n}{r}}\det(p.w-p.0)^{-\frac{n}{r}}\det(p.x-p.w)^{-\frac{n}{r}} \notag\\*
&=\biggl(\frac{\chi(\kappa(p,x))\chi(\kappa(p,0))\det(x-0)}{\chi(\kappa(p,w))\chi(\kappa(p,0))\det(w-0)\chi(\kappa(p,x))\chi(\kappa(p,w))\det(x-w)}\biggr)^{\frac{n}{r}} \notag\\*
&=\chi(\kappa(p,w))^{-\frac{2n}{r}}
\biggl(\frac{\det(x)}{\det(w)\det(x-w)}\biggr)^{\frac{n}{r}}. \label{formula_compute_P}
\end{align}}
Then since $\exp(\fp)\times D\times D$ is simply-connected, 
\begin{multline*}
\exp(\fp)\times D\times D\ni (p,x,w) \\
\longmapsto \sigma(\kappa(p,x)),\;\sigma(\kappa(p,w)^{-1}),\;\tau(\kappa(p,0)),\;\tau(\kappa(p,w)^{-1}),\;\chi(\kappa(p,w))^{-\frac{2n}{r}}
\end{multline*}
are single-valued. Hence it suffices to show that the functions 
\begin{equation}
X_1\ni (x,w)
\longmapsto \sigma(P(x)P(w)^{-1}),\; \tau(P(x)P(x-w)^{-1}),\; 
\biggl(\frac{\det(x)}{\det(w)\det(x-w)}\biggr)^{\frac{n}{r}} \label{formula_singleval2}
\end{equation}
are single-valued. To do this, it suffices to verify that the functions 
\begin{multline*}
\Ad_G(K)\times X_2\ni (k,u,z) \\
\longmapsto \sigma(P(k.u)P(k.z)^{-1}),\; \tau(P(k.u)P(k.(u-z))^{-1}),\; 
\biggl(\frac{\det(k.u)}{\det(k.z)\det(k.(u-z))}\biggr)^{\frac{n}{r}} 
\end{multline*}
are single-valued and factor the surjective map 
\[ \Ad_G(K)\times X_2\longrightarrow X_1, \qquad (k,u,z)\longmapsto (k.u,k.z). \]
Indeed, $\sigma,\tau$ are of the form $\sigma=\chi^{-\lambda}\otimes\sigma_0$, $\tau=\chi^{-\mu}\otimes\tau_0$ for some $\lambda,\mu\in\BC$ and for some representations $\sigma_0,\tau_0$ of $K$. 
Then for $(u,z)\in X_2$ and for $k\in \widetilde{K}$ in the universal covering group of $K$, we have 
\begin{align*}
\sigma(P(k.u)P(k.z)^{-1})&=\sigma(kP(u)P(z)^{-1}k^{-1}) \\*
&=\det(u)^{-\lambda}\det(z)^\lambda\sigma_0(kP(u)P(z)^{-1}k^{-1}), \\
\tau(P(k.u)P(k.(u-z))^{-1})&=\tau(kP(u)P(u-z)^{-1}k^{-1}) \\*
&=\det(u)^{-\mu}\det(u-z)^\mu\tau_0(kP(u)P(u-z)^{-1}k^{-1}), \\
\biggl(\frac{\det(k.u)}{\det(k.z)\det(k.(u-z))}\biggr)^{\frac{n}{r}}
&=\biggl(\frac{\chi(k)^2\det(u)}{\chi(k)^4\det(z)\det(u-z)}\biggr)^{\frac{n}{r}} \\*
&=\chi(k)^{-\frac{2n}{r}}\biggl(\frac{\det(u)}{\det(z)\det(u-z)}\biggr)^{\frac{n}{r}}. 
\end{align*}
Since $X_2$ is simply-connected and $\frac{2n}{r}\in\BZ_{>0}$, these are single-valued on $K\times X_2$, 
and since the elements of the kernel of the covering map $K\to \Ad_G(K)$ commute with $P(u)P(z)^{-1}$ and $P(u)P(u-z)^{-1}$, these factor $\Ad_G(K)\times X_2$. 
Also, for $(x,w)\in X_1$, if $(x,w)=(k_1.u_1,k_1.z_1)=(k_2.u_2,k_2.z_2)$ with $(u_j,z_j)\in X_2$, $k_j\in K$, 
then since $k_2^{-1}k_1\in K_L$ holds by Lemma \ref{lem_KL}, 
we have $\chi(k_1)=\chi(k_2)$, $\det(u_1)=\det(u_2)$, $\det(z_1)=\det(z_2)$ and $\det(u_1-z_1)=\det(u_2-z_2)$, and thus the above functions factor $X_1$. 
Hence the functions in (\ref{formula_singleval2}) are single-valued on $X_1$, and so are (\ref{formula_singleval1}). 
\end{proof}

\subsection{Proof of continuity in $L^\infty_{\mathrm{loc}}$}

In this subsection, we prove the following lemma. 

\begin{lemma}\label{lem_conti}
Under the setting of Theorem \ref{main_thm}, the map 
\[ \cF^{\sigma,\tau}_{\rho,\rK\uparrow}\colon L^\infty_{\mathrm{loc}}(D,V_\rho) \longrightarrow L^\infty_{\mathrm{loc}}(D\times D,V_\sigma\otimes V_\tau)\det(x-y)^{m-\lceil(\lambda_1+\mu_1-\lambda_r-\mu_r)/2\rceil} \]
is continuous. 
\end{lemma}

\begin{proof}
Let $m-\lceil(\lambda_1+\mu_1-\lambda_r-\mu_r)/2\rceil=:m'$, let $\det(z)^{-m}\rK(z)=:\rK'(z)$, and let 
\[ C_{\rK'}:=\sup_{z\in D}|\rK'(z)|_{\op,\rho\to\sigma\otimes\tau}. \]
For $y\in D$, let $p_y\in\exp(\fp)\subset G$ be the unique element satisfying $p_y.0=y$. Then $y\mapsto p_y$ is continuous. 
We take an arbitrary relatively compact open set $Y\subset D\times D$, and let $C_{Y,m'}^{\sigma,\tau}, C_Y^\rho>0$ be the constants given by 
\begin{align*}
C_{Y,m'}^{\sigma,\tau}&:=\sup_{(x,y)\in Y}|\chi(\kappa(p_y,p_y^{-1}.x))\chi(\kappa(p_y,0))|^{-m'}
\bigl|\sigma(\kappa(p_y,p_y^{-1}.x))\otimes\tau(\kappa(p_y,0))\bigr|_{\op,\sigma\otimes\tau}, \\
C_Y^\rho&:=\sup_{(x,y)\in Y}\sup_{|w|_\infty\le |p_y^{-1}.x|_\infty}\bigl|\rho(\kappa(p_y,w)^{-1})\bigr|_{\op,\rho}. 
\end{align*}
Also, we set 
\[ Y':=\bigcup_{(x,y)\in Y\cap (D\times D)^\times}C(x,y)\subset D. \]
Then $Y'$ is bounded with respect to the $G$-invariant metric. Indeed, $Y\subset D\times D$ is bounded, 
namely, $Y\subset B_{x_0}(r)\times B_{x_0}(r)$ holds for some $x_0\in D$ and $r>0$, and by Lemma \ref{lem_bdd}, we have 
\[ Y'=\bigcup_{(x,y)\in Y\cap(D\times D)^\times}C(x,y)\subset \bigcup_{x,y\in B_{x_0}(r)} B_y(\dist(x,y))
\subset \bigcup_{y\in B_{x_0}(r)} B_y(2r)\subset B_{x_0}(3r). \]
Hence $Y'$ is relatively compact in $D$. For $f\in\cO(D,V_\rho)$, we set 
\[ \Vert f\Vert_{Y',\rho}:=\sup_{w\in Y'}|f(w)|_\rho. \]
Then for $(x,y)\in Y\cap(D\times D)^\times$, by putting $u:=p_y^{-1}.x\in D$, we have 
\begin{align*}
&|\det(x-y)^{-m'}(\cF^{\sigma,\tau}_{\rho,\rK\uparrow} f)(x,y)|_{\sigma\otimes\tau} \\
&\le |\det(x-y)|^{-m'+\frac{n}{r}}\int_{C(x,y)} |\det(w-y)|^{-\frac{n}{r}}|\det(x-w)|^{-\frac{n}{r}} \\*
&\eqspace{}\times \bigl|(\sigma\otimes\tau)(P(x-y))\bigl(\sigma(P(w-y)^{-1})\otimes \tau(P(x-w)^{-1})\bigr) \\*
&\eqspace{}\times\rK\bigl(((w-y)^{-1}+(x-w)^{-1})^{-1}\bigr)\bigr|_{\op,\rho\to\sigma\otimes\tau} |f(w)|_\rho\,|dw| \\
&\le \Vert f\Vert_{Y',\rho}|\det(p_y.u-p_y.0)|^{-m'+\frac{n}{r}}\int_{C(u,0)} |\det(p_y.z-p_y.0)|^{-\frac{n}{r}}|\det(p_y.u-p_y.z)|^{-\frac{n}{r}} \\*
&\eqspace{}\times \bigl|\bigl(\sigma(P(p_y.u-p_y.0)P(p_y.z-p_y.0)^{-1})\otimes \tau(P(p_y.u-p_y.0)P(p_y.u-p_y.z)^{-1})\bigr) \\*
&\eqspace{}\times \rK\bigl(((p_y.z-p_y.0)^{-1}+(p_y.u-p_y.z)^{-1})^{-1}\bigr)\bigr|_{\op,\rho\to\sigma\otimes\tau}|\chi(\kappa(p_y,z))|^{\frac{2n}{r}} \,|dz|. 
\end{align*}
Then by Lemmas \ref{lem_cov_det}, \ref{lem_cov_inv}, as in (\ref{formula_compute_P}), we have 
\begin{gather*}
P(p_y.u-p_y.0)P(p_y.z-p_y.0)^{-1}
=\kappa(p_y,u)P(u)P(z)^{-1}\kappa(p_y,z)^{-1}, \\
P(p_y.u-p_y.0)P(p_y.u-p_y.z)^{-1}
=\kappa(p_y,0)P(u)P(z-u)^{-1}\kappa(p_y,z)^{-1}, \\
\begin{split}
&\det(p_y.u-p_y.0)^{-1}\det(p_y.z-p_y.0)\det(p_y.u-p_y.z) \\
&=\det(u)^{-1}\det(z)\det(u-z)\chi(\kappa(p_y,z))^2, 
\end{split}\\
\begin{split}
&\bigl(\sigma(\kappa(p_y,z)^{-1})\otimes\tau(\kappa(p_y,z)^{-1})\bigr)
\rK\bigl(((p_y.z-p_y.0)^{-1}+(p_y.u-p_y.z)^{-1})^{-1}\bigr) \\
&=\bigl(\sigma(\kappa(p_y,z)^{-1})\otimes\tau(\kappa(p_y,z)^{-1})\bigr)
\rK\bigl(\kappa(p_y,z).((z-0)^{-1}+(u-z)^{-1})^{-1}\bigr) \\
&=\rK\bigl((z^{-1}+(u-z)^{-1})^{-1}\bigr)\rho(\kappa(p_y,z)^{-1}). 
\end{split}
\end{gather*}
Hence we have 
\begin{align*}
&|\det(x-y)^{-m'}(\cF^{\sigma,\tau}_{\rho,\rK\uparrow} f)(x,y)|_{\sigma\otimes\tau} \\
&\le \Vert f\Vert_{Y',\rho}|\chi(\kappa(p_y,u))\chi(\kappa(p_y,0))|^{-m'}|\det(u)|^{-m'+\frac{n}{r}}\int_{C(u,0)} |\det(z)|^{-\frac{n}{r}}|\det(u-z)|^{-\frac{n}{r}} \\*
&\eqspace{}\times \bigl|\bigl(\sigma(\kappa(p_y,u)P(u)P(z)^{-1})\otimes \tau(\kappa(p_y,0)P(u)P(u-z)^{-1})\bigr) \\*
&\eqspace{}\times\rK\bigl((z^{-1}+(u-z)^{-1})^{-1}\bigr)\rho(\kappa(p_y,z)^{-1})\bigr|_{\op,\rho\to\sigma\otimes\tau} \,|dz| \\
&\le C_{Y,m'}^{\sigma,\tau}C_Y^\rho \Vert f\Vert_{Y',\rho}|\det(u)|^{-m'+\frac{n}{r}}\int_{C(u,0)} |\det(z)|^{-\frac{n}{r}}|\det(u-z)|^{-\frac{n}{r}} \\*
&\eqspace{}\times \bigl|\bigl(\sigma(P(u)P(z)^{-1})\otimes \tau(P(u)P(u-z)^{-1})\bigr)
\rK\bigl((z^{-1}+(u-z)^{-1})^{-1}\bigr)\bigr|_{\op,\rho\to\sigma\otimes\tau} \,|dz| \\
&=C_{Y,m'}^{\sigma,\tau}C_Y^\rho \Vert f\Vert_{Y',\rho}|\det(u)|^{-m-m'+\frac{n}{r}}\int_{C(u,0)} |\det(z)|^{m-\frac{n}{r}}|\det(u-z)|^{m-\frac{n}{r}} \\*
&\eqspace{}\times \bigl|\bigl(\sigma(P(u)P(z)^{-1})\otimes \tau(P(u)P(u-z)^{-1})\bigr)
\rK'\bigl((z^{-1}+(u-z)^{-1})^{-1}\bigr)\bigr|_{\op,\rho\to\sigma\otimes\tau} \,|dz|, 
\end{align*}
where at the last equality, we have used the special case of (\ref{formula_det_inv}), 
\begin{equation}\label{formula_det_inv2}
\det\bigl((z^{-1}+(u-z)^{-1})^{-1}\bigr)=\det(z)\det(u)^{-1}\det(u-z). 
\end{equation}
Next, for a fixed Jordan frame $\{e_1,\ldots,e_r\}\subset\fn^+$, let 
\[ \fa^+_{++}:=\{a_1e_1+\cdots+a_re_r\in\fn^+\mid a_j\in\BR,\;a_1\ge \cdots\ge a_r>0\}\subset\Omega, \]
and we take $k\in K$ such that $u':=k^{-1}.u\in\fa^+_{++}\subset\Omega$ holds. Then for $z\in C(u,0)$, we have $z':=k^{-1}.z\in C(u',0)=\Omega\cap(u'-\Omega)$. 
Since 
\begin{align*}
z'-(z^{\prime-1}+(u'-z')^{-1})^{-1}&=P(z^{\prime 1/2})\bigl(e-(e+(P(z^{\prime -1/2})u-e)^{-1})^{-1}\bigr) \\
&=P(z^{\prime 1/2})(P(z^{\prime -1/2})u')^{-1}=P(z')u^{\prime-1}\in\Omega 
\end{align*}
holds, we have 
\[ (z^{-1}+(u-z)^{-1})^{-1}=k.(z^{\prime -1}+(u'-z')^{-1})^{-1}\in k.(\Omega\cap(z'-\Omega))\subset D, \]
so that 
\[ \bigl|\rK'\bigl((z^{-1}+(u-z)^{-1})^{-1}\bigr)\bigr|_{\op,\rho\to\sigma\otimes\tau}
\le \sup_{w\in D}|\rK'(w)|_{\op,\rho\to\sigma\otimes\tau}=C_{\rK'} \]
holds. Therefore, for $(x,y)\in Y\cap(D\times D)^\times$, we get 
\begin{align*}
&|\det(x-y)^{-m'}(\cF^{\sigma,\tau}_{\rho,\rK\uparrow} f)(x,y)|_{\sigma\otimes\tau} \\
&\le C_{Y,m'}^{\sigma,\tau}C_Y^\rho C_{\rK'} \Vert f\Vert_{Y',\rho}|\det(k.u')|^{-m-m'+\frac{n}{r}}\int_{\Omega\cap(u'-\Omega)} |\det(k.z')|^{m-\frac{n}{r}}|\det(k.u'-k.z')|^{m-\frac{n}{r}} \\*
&\eqspace{}\times \bigl|\sigma(kP(u')P(z')^{-1}k^{-1})\otimes \tau(kP(u')P(u'-z')^{-1}k^{-1})\bigr|_{\op,\sigma\otimes\tau} \,dz' \\
&=C_{Y,m'}^{\sigma,\tau}C_Y^\rho C_{\rK'} \Vert f\Vert_{Y',\rho}\det(u')^{-m-m'+\frac{n}{r}}\int_{\Omega\cap(u'-\Omega)} \det(z')^{m-\frac{n}{r}}\det(u'-z')^{m-\frac{n}{r}} \\*
&\eqspace{}\times \bigl|\sigma(P(u')P(z')^{-1})\otimes \tau(P(u')P(u'-z')^{-1})\bigr|_{\op,\sigma\otimes\tau} \,dz' \\
&=C_{Y,m'}^{\sigma,\tau}C_Y^\rho C_{\rK'} \Vert f\Vert_{Y',\rho}\det(u')^{-m-m'+\frac{n}{r}}\!\!\int_{\Omega\cap(e-\Omega)} \hspace{-9pt}\det(P(u^{\prime 1/2})z)^{m-\frac{n}{r}}\det(P(u^{\prime 1/2})(e\hspace{-1pt}-\hspace{-1pt}z))^{m-\frac{n}{r}} \\*
&\eqspace{}\times \bigl|\sigma(P(u')P(P(u^{\prime 1/2})z)^{-1})\otimes \tau(P(u')P(P(u^{\prime 1/2})(e-z))^{-1})\bigr|_{\op,\sigma\otimes\tau}\det(u')^{\frac{n}{r}} \,dz \\
&=C_{Y,m'}^{\sigma,\tau}C_Y^\rho C_{\rK'} \Vert f\Vert_{Y',\rho}\det(u')^{m-m'}\int_{\Omega\cap(e-\Omega)} \det(z)^{m-\frac{n}{r}}\det(e-z)^{m-\frac{n}{r}} \\*
&\eqspace{}\times \bigl|\sigma(P(u^{\prime 1/2})P(z)^{-1}P(u^{\prime -1/2}))\otimes \tau(P(u^{\prime 1/2})P(e-z)^{-1}P(u^{\prime -1/2}))\bigr|_{\op,\sigma\otimes\tau} \,dz \\
&\le C_{Y,m'}^{\sigma,\tau}C_Y^\rho C_{\rK'} \Vert f\Vert_{Y',\rho}
\bigl|(\sigma\otimes\tau)(P(u^{\prime 1/2}))\bigr|_{\op,\sigma\otimes\tau}\bigl|(\sigma\otimes\tau)(P(u^{\prime -1/2}))\bigr|_{\op,\sigma\otimes\tau}\det(u')^{m-m'} \\*
&\eqspace{}\times \int_{\Omega\cap(e-\Omega)} \det(z)^{m-\frac{n}{r}}\det(e-z)^{m-\frac{n}{r}}\bigl|\sigma(P(z)^{-1})\otimes \tau(P(e-z)^{-1})\bigr|_{\op,\sigma\otimes\tau} \,dz.  
\end{align*}
Now, if $z=k.a$ with $k\in K_L$, $a=\sum_{j=1}^r a_je_j\in\fa^+_{++}\cap(e-\Omega)$, then we have 
\[ \bigl|\sigma(P(z)^{-1})\otimes \tau(P(e-z)^{-1})\bigr|_{\op,\sigma\otimes\tau}
=\prod_{j=1}^r a_j^{\Re\lambda_j}(1-a_{r-j+1})^{\Re\mu_j}, \]
and by \cite[Theorem VI.2.3]{FK}, if $\Re\lambda_r+m>\frac{n}{r}-1$, $\Re\mu_r+m>\frac{n}{r}-1$, then there exists a constant $C>0$ such that 
\begin{align*}
B_{\sigma,\tau,m}:\hspace{-3pt}&=\int_{\Omega\cap(e-\Omega)} \det(z)^{m-\frac{n}{r}}\det(e-z)^{m-\frac{n}{r}}\bigl|\sigma(P(z)^{-1})\otimes \tau(P(e-z)^{-1})\bigr|_{\op,\sigma\otimes\tau} \,dz \\
&=C\int_{\fa^+_{++}\cap(e-\Omega)}\prod_{j=1}^r a_j^{\Re\lambda_j+m-\frac{n}{r}}(1-a_{r-j+1})^{\Re\mu_j+m-\frac{n}{r}}\prod_{1\le i<j\le r}(a_i-a_j)^d\,da<\infty
\end{align*}
holds. Similarly, for $u'=\sum_{j=1}^ru_je_j\in\fa^+_{++}\cap(e-\Omega)$, 
we have 
\begin{align*}
&\bigl|(\sigma\otimes\tau)(P(u^{\prime 1/2}))\bigr|_{\op,\sigma\otimes\tau}\bigl|(\sigma\otimes\tau)(P(u^{\prime -1/2}))\bigr|_{\op,\sigma\otimes\tau}\det(u')^{m-m'} \\
&=\prod_{j=1}^r u_j^{-\Re(\lambda_{r-j+1}+\mu_{r-j+1})/2}\prod_{j=1}^r u_j^{\Re(\lambda_j+\mu_j)/2}\prod_{j=1}^r u_j^{\lceil\Re(\lambda_1+\mu_1-\lambda_r-\mu_r)/2\rceil}\le 1, 
\end{align*}
since $\Re\lambda_1\ge\cdots\ge \Re\lambda_r$, $\Re\mu_1\ge\cdots\ge \Re\mu_r$ and $u_j\le 1$. Therefore, for $(x,y)\in Y\cap(D\times D)^\times$, we get 
\[ |\det(x-y)^{-m'}(\cF^{\sigma,\tau}_{\rho,\rK\uparrow} f)(x,y)|_{\sigma\otimes\tau}
\le C_{Y,m'}^{\sigma,\tau}C_Y^\rho C_{\rK'} B_{\sigma,\tau,m} \Vert f\Vert_{Y',\rho}, \]
and $\cF^{\sigma,\tau}_{\rho,\rK\uparrow}\colon L^\infty_{\mathrm{loc}}(D,V_\rho)\to L^\infty_{\mathrm{loc}}(D\times D,V_\sigma\otimes V_\tau)\det(x-y)^{m'}$ is continuous. 
\end{proof}

\subsection{Proof of intertwining property on $L^\infty_{\mathrm{loc}}$}

In this subsection, we prove that $\cF^{\sigma,\tau}_{\rho,\rK\uparrow}$ intertwines the $\widetilde{G}$-action on $L^\infty_{\mathrm{loc}}$-spaces. 

\begin{lemma}\label{lem_intertwining}
For $g\in \widetilde{G}$, we have 
$(\hat{\sigma}\hotimes\hat{\tau})(g)\circ\cF^{\sigma,\tau}_{\rho,\rK\uparrow}
=\cF^{\sigma,\tau}_{\rho,\rK\uparrow}\circ\hat{\rho}(g)$ on $L^\infty_{\mathrm{loc}}$. 
\end{lemma}

\begin{proof}
By Lemmas \ref{lem_cov_det}, \ref{lem_cov_inv}, we get 
\begin{align*}
&(((\hat{\sigma}\hotimes\hat{\tau})(g^{-1})\circ\cF^{\sigma,\tau}_{\rho,\rK\uparrow})f)(x,y) \\
&=\bigl(\sigma(\kappa(g,x))^{-1}\otimes\tau(\kappa(g,y))^{-1}\bigr)\det(g.x-g.y)^{\frac{n}{r}}\bigl(\sigma(P(g.x-g.y))\otimes\tau(P(g.x-g.y))\bigr) \\*
&\eqspace{}\times \int_{C(g.x,g.y)} \det(w-g.y)^{-\frac{n}{r}}\det(g.x-w)^{-\frac{n}{r}}\bigl(\sigma(P(w-g.y)^{-1})\otimes\tau(P(g.x-w)^{-1})\bigr) \\*
&\eqspace{}\times \rK\bigl(((w-g.y)^{-1}+(g.x-w)^{-1})^{-1}\bigr) f(w)\,dw \\
&=\chi(\kappa(g,x))^{\frac{n}{r}}\chi(\kappa(g,y))^{\frac{n}{r}}\det(x-y)^{\frac{n}{r}}\bigl(\sigma(P(x-y){}^t\hspace{-1pt}\kappa(g,y))\otimes\tau(P(x-y){}^t\hspace{-1pt}\kappa(g,x))\bigr) \\*
&\eqspace{}\times \int_{C(x,y)} \det(g.z-g.y)^{-\frac{n}{r}}\det(g.x-g.z)^{-\frac{n}{r}}\bigl(\sigma(P(g.z-g.y)^{-1})\otimes\tau(P(g.x-g.z)^{-1})\bigr) \\*
&\eqspace{}\times \rK\bigl(((g.z-g.y)^{-1}+(g.x-g.z)^{-1})^{-1}\bigr) f(g.z)\chi(\kappa(g,z))^{\frac{2n}{r}}\,dz \\
&=\det(x-y)^{\frac{n}{r}}\bigl(\sigma(P(x-y))\otimes\tau(P(x-y))\bigr) \\*
&\eqspace{}\times \int_{C(x,y)} \det(z-y)^{-\frac{n}{r}}\det(x-z)^{-\frac{n}{r}}\bigl(\sigma(P(z-y)^{-1})\otimes\tau(P(x-z)^{-1})\bigr) \\*
&\eqspace{}\times \bigl(\sigma(\kappa(g,z)^{-1})\otimes\tau(\kappa(g,z)^{-1})\bigr)\rK\bigl(\kappa(g,z).((z-y)^{-1}+(x-z)^{-1})^{-1}\bigr) f(g.z)\,dz \\
&=\det(x-y)^{\frac{n}{r}}(\sigma\otimes\tau)(P(x-y))\int_{C(x,y)} \det(z-y)^{-\frac{n}{r}}\det(x-z)^{-\frac{n}{r}} \\*
&\eqspace{}\times\hspace{-1pt} \bigl(\sigma(P(z-y)^{-1})\otimes\tau(P(x-z)^{-1})\bigr)\rK\bigl(((z-y)^{-1}+(x-z)^{-1})^{-1}\bigr)\rho(\kappa(g,z)^{-1})f(g.z)\hspace{1pt}dz \\
&=((\cF^{\sigma,\tau}_{\rho,\rK\uparrow}\circ\hat{\rho}(g^{-1}))f)(x,y). \qedhere
\end{align*}
\end{proof}

\subsection{Proof of holomorphy}\label{subsection_hol}

In this subsection, we prove that the image of $\cF^{\sigma,\tau}_{\rho,\rK\uparrow}$ becomes a holomorphic function. 

\begin{lemma}\label{lem_holomorphic}
Under the setting of Theorem \ref{main_thm}, we have 
\[ \cF^{\sigma,\tau}_{\rho,\rK\uparrow}(\cO(D,V_\rho))\subset \cO(D\times D,V_\sigma\otimes V_\tau)\det(x-y)^{m-\lceil(\lambda_1+\mu_1-\lambda_r-\mu_r)/2\rceil}. \]
\end{lemma}

First, we verify that $(\cF^{\sigma,\tau}_{\rho,\rK\uparrow}f)(x,y)$ is symmetric with respect to $x$ and $y$, so that it is enough to verify the holomorphy for one variable. 

\begin{lemma}\label{lem_symmetric}
Under the setting of Theorem \ref{main_thm}, for $f\in\cO(D,V_\rho)$, $(x,y)\in(D\times D)^\times$, we have 
\[ (\cF^{\sigma,\tau}_{\rho,\rK\uparrow}f)(x,y)=(-1)^{\deg\rK} (\cF^{\tau,\sigma}_{\rho,\rK\uparrow}f)(y,x) \]
under the identification $V_\sigma\otimes V_\tau\simeq V_\tau\otimes V_\sigma$,
where $\deg\rK$ is the degree of the polynomial $\rK$. 
\end{lemma}

\begin{proof}
First, we note that $\rK$ is a homogeneous polynomial since $\rK$ is invariant under the action of the center of $\widetilde{K}$. 
By Lemma \ref{lem_ori_contour}\,(2), we have 
\begin{align*}
&(\cF^{\sigma,\tau}_{\rho,\rK\uparrow} f)(x,y) \\
&=\det(x-y)^{\frac{n}{r}}\int_{C(x,y)} \det(w-y)^{-\frac{n}{r}}\det(x-w)^{-\frac{n}{r}} \bigl(\sigma(P(x-y)P(w-y)^{-1})\\*
&\eqspace{}\otimes \tau(P(x-y)P(x-w)^{-1})\bigr)\rK\bigl(((w-y)^{-1}+(x-w)^{-1})^{-1}\bigr) f(w)\,dw \\
&=\det(-e)^{\frac{n}{r}}\det(y-x)^{\frac{n}{r}}\int_{C(x,y)} \det(y-w)^{-\frac{n}{r}}\det(w-x)^{-\frac{n}{r}}\bigl(\sigma(P(y-x)P(y-w)^{-1}) \\*
&\eqspace{}\otimes \tau(P(y-x)P(w-x)^{-1})\bigr)\rK\bigl(-((y-w)^{-1}+(w-x)^{-1})^{-1}\bigr) f(w)\,dw \\
&=(-1)^{n+\deg\rK}\det(y-x)^{\frac{n}{r}}(\sigma\otimes\tau)(P(y-x))\int_{C(y,x)} \det(y-w)^{-\frac{n}{r}}\det(w-x)^{-\frac{n}{r}} \\*
&\eqspace{}\times \bigl(\sigma(P(y-w)^{-1})\otimes \tau(P(w-x)^{-1})\bigr)\rK\bigl(((y-w)^{-1}+(w-x)^{-1})^{-1}\bigr) f(w)(-1)^n\,dw \\
&=(-1)^{\deg\rK}(\cF^{\tau,\sigma}_{\rho,\rK\uparrow} f)(y,x). \qedhere
\end{align*}
\end{proof}

Next, we verify the following. 

\begin{lemma}\label{lem_diffeo_m}
We have a diffeomorphism 
\[ \widetilde{\opm}\colon((K\times L)/\diag(K_L)) \times \sqrt{-1}\fk_\fl\longrightarrow K^\BC, \qquad 
((k,l)\diag(K_L),b)\longmapsto kl^{-1}\exp(b). \]
\end{lemma}

\begin{proof}
First, we prove the surjectivity of $\widetilde{\opm}$. Let $g\in K^\BC$, and consider $g.e\in \fp^+$. 
Then we can take $k\in K$ such that $k^{-1}.(g.e)\in\Omega$, and $l'\in L$ such that $l'.(k^{-1}.(g.e))=e$. 
Let $a:=l'k^{-1}g\in K^\BC$. Then since $a.e=e$, we have $a\in K_L^\BC=K_L\exp(\sqrt{-1}\fk_\fl)$, and we can take $k_a\in K_L$ and $b\in\sqrt{-1}\fk_\fl$ such that $a=k_a\exp(b)$. 
Then by putting $l:=k_a^{-1}l'\in L$, we get $g=kl^{-1}\exp(b)$. 

Next, we prove the injectivity of $\widetilde{\opm}$. 
Suppose that $k_j\in K$, $l_j\in L$, $b_j\in \sqrt{-1}\fk_\fl$ $(j=1,2)$ satisfy $k_1l_1^{-1}\exp(b_1)=k_2l_2^{-1}\exp(b_2)$. 
Then since $\exp(b_j).e=e$ holds, we have $l_1^{-1}.e=k_1^{-1}k_2l_2^{-1}.e$, and since $k_1^{-1}k_2\in K$, $l_j^{-1}.e\in\Omega$, we have $k_1^{-1}k_2\in K_L\subset L$ by Lemma \ref{lem_KL},  
and hence $k_0:=l_1k_1^{-1}k_2l_2^{-1}=\exp(b_1)\exp(b_2)^{-1}\in L\cap K_L^\BC=K_L$. Then by $\exp(b_1)=k_0\exp(b_2)$ and the injectivity of $K_L\times\exp(\sqrt{-1}\fk_\fl)\to K_L^\BC$, we have $k_0=I$, $b_1=b_2$, and get $k_1l_1^{-1}=k_2l_2^{-1}$. 
Then since $k_1^{-1}k_2=l_1^{-1}l_2\in K\cap L=K_L$ holds, we have $(k_1,l_1)\diag(K_L)=(k_2,l_2)\diag(K_L)$. 

Finally, we prove the surjectivity of the differential of 
\[ \opm\colon K\times L\times \sqrt{-1}\fk_\fl\longrightarrow K^\BC, \qquad 
(k,l,b)\longmapsto kl^{-1}\exp(b). \]
at each point. 
We take any $k_0\in K$, $l_0\in L$ and $b_0\in\sqrt{-1}\fk_\fl$. 
We consider the translations 
\begin{gather*}
\rL(k_0,l_0,b_0)\colon K\times L\times \sqrt{-1}\fk_\fl\longrightarrow K\times L\times\sqrt{-1}\fk_\fl, \qquad
(k,l,b)\longmapsto (k_0k,l_0l,b_0+b), \\
\rL(l_0k_0^{-1})\rR(\exp(-b_0))\colon K^\BC\longrightarrow K^\BC, \qquad
g\longmapsto l_0k_0^{-1}g\exp(-b_0), 
\end{gather*}
and let 
\[ \opm_{(k_0,l_0,b_0)}:=\rL(l_0k_0^{-1})\rR(\exp(-b_0))\circ\opm\circ \rL(k_0,l_0,b_0). \]
Then the surjectivity of the differential of $\opm$ at $(k_0,l_0,b_0)$, 
\[ (d\opm)_{(k_0,l_0,b_0)}\colon T_{(k_0,l_0,b_0)}(K\times L\times \sqrt{-1}\fk_\fl)
\longrightarrow T_{k_0l_0^{-1}\exp(b_0)}K^\BC \]
is equivalent to the surjectivity of the differential of $\opm_{(k_0,l_0,b_0)}$ at $(I,I,0)$, 
\begin{gather*}
(d\opm_{(k_0,l_0,b_0)})_{(I,I,0)}\colon T_{(I,I,0)}(K\times L\times \sqrt{-1}\fk_\fl)=\fk\oplus\fl\oplus\sqrt{-1}\fk_\fl
\longrightarrow T_{I}K^\BC=\fk^\BC, \\
\begin{split}
(k,l,b)\longmapsto{} &\frac{d}{dt}\biggr|_{t=0}l_0k_0^{-1}k_0\exp(tk)\exp(-tl)l_0^{-1}\exp(b_0+tb)\exp(-b_0) \\
&=\Ad(l_0)(k-l)+\frac{d}{dt}\biggr|_{t=0}\exp(b_0+tb)\exp(-b_0). 
\end{split}
\end{gather*}
Now, since $K_L\times\exp(\sqrt{-1}\fk_\fl)\to K_L^\BC$ is a diffeomorphism, 
\[ \fk_\fl\oplus\sqrt{-1}\fk_\fl\longrightarrow \fk_\fl^\BC, \qquad 
(k,b)\longmapsto k+\frac{d}{dt}\biggr|_{t=0}\exp(b_0+tb)\exp(-b_0) \]
is bijective, and since $\Ad(l_0)(\fl)=\fl\supset\fk_\fl$, we have 
\begin{align*}
(d\opm_{(k_0,l_0,b_0)})_{(I,I,0)}(\fk\oplus\fl\oplus\sqrt{-1}\fk_\fl)
&=\Ad(l_0)\fk+\fl+\sqrt{-1}\fk_\fl \\
&=\Ad(l_0)(\fk_\fl+\sqrt{-1}\fp_\fl)+(\fk_\fl+\fp_\fl)+\sqrt{-1}\fk_\fl \\
&=\fk_\fl+\fp_\fl+\sqrt{-1}(\fk_\fl+\Ad(l_0)\fp_\fl), 
\end{align*}
where $\fl=\fk_\fl+\fp_\fl$ is the Cartan decomposition of $\fl$. 
In addition, we have $\fk_\fl\cap \Ad(l_0)\fp_\fl=\{0\}$, since the Killing form of $\fl$ is positive semidefinite on $\Ad(l_0)\fp_\fl$ and negative definite on $\fk_\fl$. 
Combining this with the injectivity of $\Ad(l_0)|_{\fp_\fl}\colon \fp_\fl\to\fk_\fl+\fp_\fl$, we get $\fk_\fl+\Ad(l_0)\fp_\fl=\fk_\fl+\fp_\fl$, and 
\begin{align*}
(d\opm_{(k_0,l_0,b_0)})_{(I,I,0)}(\fk\oplus\fl\oplus\sqrt{-1}\fk_\fl)
=\fk_\fl+\fp_\fl+\sqrt{-1}(\fk_\fl+\fp_\fl)=\fk^\BC. 
\end{align*}
Hence the differential of $\opm$ is surjective everywhere, and $\widetilde{\opm}$ is a diffeomorphism. 
\end{proof}

\begin{proof}[Proof of Lemma \ref{lem_holomorphic}]
First, since $\det(x-y)^{-m+\lceil(\lambda_1+\mu_1-\lambda_r-\mu_r)/2\rceil}(\cF^{\sigma,\tau}_{\rho,\rK\uparrow}f)(x,y)$ is locally bounded on $D\times D$ by Lemma \ref{lem_conti}, it suffices to verify $(\cF^{\sigma,\tau}_{\rho,\rK\uparrow}f)(x,y)\in \cO((D\times D)^\times, V_\sigma\otimes V_\tau)$, 
and by Lemma \ref{lem_symmetric} and Hartogs' theorem, it is enough to verify that $(\cF^{\sigma,\tau}_{\rho,\rK\uparrow} f)(x,y_0)\in\cO(\{x\in D\mid\det(x-y_0)\ne 0\},V_\sigma\otimes V_\tau)$ holds for each fixed $y_0\in D$. 
Moreover, for each $y_0=g^{-1}.0\in D$ with $g\in \widetilde{G}$, by Lemma \ref{lem_intertwining}, we have 
\begin{align*}
&(\cF^{\sigma,\tau}_{\rho,\rK\uparrow} f)(x,y_0)=(((\hat{\sigma}\hotimes\hat{\tau})(g^{-1})\circ\cF^{\sigma,\tau}_{\rho,\rK\uparrow} \circ\hat{\rho}(g))f)(x,y_0) \\
&=\bigl(\sigma(\kappa(g,x))^{-1}\otimes\tau(\kappa(g,y_0))^{-1}\bigr)((\cF^{\sigma,\tau}_{\rho,\rK\uparrow} \circ\hat{\rho}(g))f)(g.x,0), 
\end{align*}
and hence it suffices to consider the case $y_0=0$. 

In the following, we prove $(\cF^{\sigma,\tau}_{\rho,\rK\uparrow} f)(x,0) \in\cO(D^\times,V_\sigma\otimes V_\tau)$, where $D^\times$ is as in (\ref{formula_Dtimes}). We set 
\[ L^-:=\{a\in L\mid a.e\in D\cap\Omega=\Omega\cap(e-\Omega)\}, \]
so that 
\[ K^\BC/K_L^\BC\supset KL^-K_L^\BC/K_L^\BC\overset{\sim}{\longrightarrow} D^\times\subset\fp^{+\times} \]
hold by $a K_L^\BC\mapsto a.e$. Then by Lemma \ref{lem_diffeo_m}, the multiplication map 
\[ KL^-\times \exp(\sqrt{-1}\fk_\fl)\longrightarrow K^\BC, \qquad (a,b)\longmapsto ab \]
is a diffeomorphism onto its image. Let $\cU\subset K^\BC$ be an open subset satisfying 
\begin{itemize}
\item $KL^-\subset\cU\subset\{a\in K^\BC\mid a.(\overline{\Omega}\cap(e-\overline{\Omega}))\subset D\}\subset K^\BC$, 
\item For all $a\in KL^-$, $\cU_a:=\cU\cap a\exp(\sqrt{-1}\fk_\fl)\subset\cU$ is path-connected, and 
\item $\cU=\bigsqcup_{a\in KL^-}\cU_a$, 
\end{itemize}
so that $\cU K_L^\BC/K_L^\BC\simeq D^\times$ holds, where $\overline{\Omega}$ is the closure of $\Omega$. 
For $f\in\cO(D,V_\rho)$, we define the function $F\colon\cU\to V_\sigma\otimes V_\tau$ by 
\begin{align*}
F(a)&=\det(a.e)^{\frac{n}{r}}(\sigma\otimes\tau)(P(a.e))\int_{a.(\Omega\cap(e-\Omega))} \det(w)^{-\frac{n}{r}}\det(a.e-w)^{-\frac{n}{r}} \\*
&\eqspace{}\times \bigl(\sigma(P(w)^{-1})\otimes \tau(P(a.e-w)^{-1})\bigr)\rK\bigl((w^{-1}+(a.e-w)^{-1})^{-1}\bigr) f(w)\,dw. 
\end{align*}
To prove $(\cF^{\sigma,\tau}_{\rho,\rK\uparrow} f)(x,0)\in\cO(D^\times,V_\sigma\otimes V_\tau)$, it suffices to show $F\in\cO(\cU,V_\sigma\otimes V_\tau)$ and $F(a)=(\cF^{\sigma,\tau}_{\rho,\rK\uparrow} f)(a.e,0)$. 
First, $F\in\cO(\cU,V_\sigma\otimes V_\tau)$ follows from 
\begin{align*}
&F(a)
=\det(a.e)^{\frac{n}{r}}(\sigma\otimes\tau)(P(a.e))\int_{\Omega\cap(e-\Omega)} \det(a.z)^{-\frac{n}{r}}\det(a.e-a.z)^{-\frac{n}{r}} \\*
&\eqspace{}\times \bigl(\sigma(P(a.z)^{-1})\otimes \tau(P(a.e\hspace{-1pt}-\hspace{-1pt}a.z)^{-1})\bigr)\rK\bigl(((a.z)^{-1}\hspace{-1pt}+\hspace{-1pt}(a.e\hspace{-1pt}-\hspace{-1pt}a.z)^{-1})^{-1}\bigr) f(a.z)\det(a.e)^{\frac{n}{r}}\hspace{1pt}dz \\
&=(\sigma\otimes\tau)(a{}^t\hspace{-1pt}a)\int_{\Omega\cap(e-\Omega)} \det(z)^{-\frac{n}{r}}\det(e-z)^{-\frac{n}{r}} \\*
&\eqspace{}\times \bigl(\sigma({}^t\hspace{-1pt}a^{-1}P(z)^{-1}a^{-1})\otimes \tau({}^t\hspace{-1pt}a^{-1}P(e-z)^{-1}a^{-1})\bigr)\rK\bigl(a.(z^{-1}+(e-z)^{-1})^{-1}\bigr) f(a.z)\,dz \\
&=(\sigma\otimes\tau)(a)\int_{\Omega\cap(e-\Omega)} \det(z)^{-\frac{n}{r}}\det(e-z)^{-\frac{n}{r}} \\*
&\eqspace{}\times \bigl(\sigma(P(z)^{-1})\otimes \tau(P(e-z)^{-1})\bigr)\rK\bigl((z^{-1}+(e-z)^{-1})^{-1}\bigr)\rho(a^{-1}) f(a.z)\,dz. 
\end{align*}
Next, for $a=kl\in KL^-\subset\cU$ with $k\in K$, $l\in L^-$, since 
\begin{align*}
a.(\Omega\cap(e-\Omega))&=kl.(\Omega\cap(e-\Omega))
=k.(\Omega\cap(l.e-\Omega))=k.C(l.e,0) \\
&=C(k.(l.e),0)=C(a.e,0) 
\end{align*}
holds, we have $F(a)=(\cF^{\sigma,\tau}_{\rho,\rK\uparrow} f)(a.e,0)$ by definition. 
Finally, for an arbitrary $a'\in\cU$, we can take $a\in KL^-$ such that $a'\in \cU_a\subset a\exp(\sqrt{-1}\fk_\fl)$. Then we have $a.e=a'.e$. 
Let $a_s\colon[0,1]\to\cU_a$ be an injective, smooth, non-singular curve such that $a_0=a$, $a_1=a'$, 
and let $V:=\bigsqcup_{0\le s\le 1}a_s.(\Omega\cap(e-\Omega))$. Then its boundary is given by 
\[ \partial V=a.(\Omega\cap(e-\Omega))\sqcup a'.(\Omega\cap(e-\Omega))\sqcup \bigsqcup_{0\le s\le 1}\partial(a_s.(\Omega\cap(e-\Omega))). \]
Since $\sigma,\tau$ are of the form $\sigma=\chi^{-\lambda}\otimes\sigma_0$, $\tau=\chi^{-\mu}\otimes\tau_0$ with continuous parameters $\lambda,\mu$, $\det(w)^{-\frac{n}{r}}\sigma(P(w)^{-1})$ vanishes on $\partial(a_s.\Omega)$ if $\Re\lambda$ is sufficiently large. 
Similarly, $\det(a.e-w)^{-\frac{n}{r}}\tau(P(a.e-w)^{-1})$ vanishes on $\partial(a_s.(e-\Omega))$ if $\Re\mu$ is sufficiently large. Then we have 
\begin{align*}
0&=\det(a.e)^{\frac{n}{r}}(\sigma\otimes\tau)(P(a.e))\int_{\partial V} \det(w)^{-\frac{n}{r}}\det(a.e-w)^{-\frac{n}{r}} \\*
&\hspace{30pt}\times \bigl(\sigma(P(w)^{-1})\otimes \tau(P(a.e-w)^{-1})\bigr)\rK\bigl((w^{-1}+(a.e-w)^{-1})^{-1}\bigr) f(w)\,dw \\
&=\det(a'.e)^{\frac{n}{r}}(\sigma\otimes\tau)(P(a'.e))\int_{a'.(\Omega\cap(e-\Omega))} \det(w)^{-\frac{n}{r}}\det(a'.e-w)^{-\frac{n}{r}} \\*
&\hspace{30pt}\times \bigl(\sigma(P(w)^{-1})\otimes \tau(P(a'.e-w)^{-1})\bigr)\rK\bigl((w^{-1}+(a'.e-w)^{-1})^{-1}\bigr) f(w)\,dw \\*
&\eqspace{}-\det(a.e)^{\frac{n}{r}}(\sigma\otimes\tau)(P(a.e))\int_{a.(\Omega\cap(e-\Omega))} \det(w)^{-\frac{n}{r}}\det(a.e-w)^{-\frac{n}{r}} \\*
&\hspace{30pt}\times \bigl(\sigma(P(w)^{-1})\otimes \tau(P(a.e-w)^{-1})\bigr)\rK\bigl((w^{-1}+(a.e-w)^{-1})^{-1}\bigr) f(w)\,dw \\
&=F(a')-F(a), 
\end{align*}
and hence we get 
\[ F(a')=F(a)=(\cF^{\sigma,\tau}_{\rho,\rK\uparrow} f)(a.e,0)=(\cF^{\sigma,\tau}_{\rho,\rK\uparrow} f)(a'.e,0). \]
By considering the analytic continuation with respect to the continuous parameters of $\sigma,\tau$, this holds if the integral converges, 
even if $\det(w)^{-\frac{n}{r}}\det(a.e-w)^{-\frac{n}{r}}\sigma(P(w)^{-1})\tau(P(a.e-w)^{-1})$ does not vanish on $\partial(a_s.(\Omega\cap(e-\Omega)))$. 
Thus we have $(\cF^{\sigma,\tau}_{\rho,\rK\uparrow} f)(x,0)\in\cO(D^\times,V_\sigma\otimes V_\tau)\simeq\cO(\cU K_L^\BC/K_L^\BC,V_\sigma\otimes V_\tau)$, and hence $(\cF^{\sigma,\tau}_{\rho,\rK\uparrow} f)(x,y)\in\cO((D\times D)^\times ,V_\sigma\otimes V_\tau)$. 
\end{proof}

By the above lemmas, we get Theorem \ref{main_thm}.

\subsection{Proof of Theorem \ref{main_thm2}}\label{subsection_thm2}

In this subsection, we prove Theorem \ref{main_thm2}. To do this, it suffices to verify $\cF^{\sigma,\tau}_{\rho,\rK\uparrow}(\cO(D,V_\rho))\subset \cO(D\times D,V_\sigma\otimes V_\tau)$. 
First, we prove that the images of constant functions are holomorphic on $D\times D$. 

\begin{lemma}\label{lem_const_func}
Under the setting of Theorem \ref{main_thm}, we regard $V_\rho\subset\cO(D,V_\rho)$ as the set of constant functions. 
\begin{enumerate}
\item For $v\in V_\rho$, the function $(\cF^{\sigma,\tau}_{\rho,\rK\uparrow}v)(x,0)$ defined for $x\in D^\times$ is holomorphically continued to $x\in\fp^{+\times}$, and for $(x,y)\in (D\times D)^\times$, we have 
\[ (\cF^{\sigma,\tau}_{\rho,\rK\uparrow}v)(x,y)=(\cF^{\sigma,\tau}_{\rho,\rK\uparrow}v)(x-y,0). \]
\item If $(\rho,V_\rho)$ satisfies the assumption of Theorem \ref{main_thm2}, 
then we have 
\[ \cF^{\sigma,\tau}_{\rho,\rK\uparrow}(V_\rho)\subset\cP(\fp^+\oplus\fp^+, V_\sigma\otimes V_\tau)\subset\cO(D\times D,V_\sigma\otimes V_\tau). \]
\end{enumerate}
\end{lemma}

\begin{proof}
(1) By Lemmas \ref{lem_intertwining} and \ref{lem_holomorphic}, we have 
\begin{align*}
\cF^{\sigma,\tau}_{\rho,\rK\uparrow}(\cdot,0)|_{V_\rho}&\in\Hom_{\widetilde{K}}\bigl(V_\rho,\cO(D,V_\sigma\otimes V_\tau)\det(x)^{m-\lceil(\lambda_1+\mu_1-\lambda_r-\mu_r)/2\rceil}\bigr) \\
&\subset \Hom_{\widetilde{K}}\bigl(V_\rho,\cO(D,V_\sigma\otimes V_\tau)[\det^{-1}]_{\widetilde{K}}\bigr)=\Hom_{\widetilde{K}}\bigl(V_\rho,\cP(\fp^+,V_\sigma\otimes V_\tau)[\det^{-1}]\bigr), 
\end{align*}
and hence $(\cF^{\sigma,\tau}_{\rho,\rK\uparrow}v)(x,0)$ is holomorphically continued to $x\in \fp^{+\times}$. Next, suppose $x,y\in D\cap\fn^+$, $x-y\in\Omega$. Then by Lemma \ref{lem_ori_contour}\,(1), we have 
\begin{align*}
&(\cF^{\sigma,\tau}_{\rho,\rK\uparrow} v)(x,y) \\
&=\det(x-y)^{\frac{n}{r}}(\sigma\otimes\tau)(P(x-y))\int_{(y+\Omega)\cap(x-\Omega)} \det(w-y)^{-\frac{n}{r}}\det(x-w)^{-\frac{n}{r}} \\
&\eqspace{}\times \bigl(\sigma(P(w-y)^{-1})\otimes \tau(P(x-w)^{-1})\bigr)\rK\bigl(((w-y)^{-1}+(x-w)^{-1})^{-1}\bigr) v\,dw \\
&=\det(x-y)^{\frac{n}{r}}(\sigma\otimes\tau)(P(x-y))\int_{\Omega\cap(x-y-\Omega)} \det(w)^{-\frac{n}{r}}\det(x-y-w)^{-\frac{n}{r}} \\
&\eqspace{}\times \bigl(\sigma(P(w)^{-1})\otimes \tau(P(x-y-w)^{-1})\bigr)\rK\bigl((w^{-1}+(x-y-w)^{-1})^{-1}\bigr) v\,dw \\
&=(\cF^{\sigma,\tau}_{\rho,\rK\uparrow} v)(x-y,0), 
\end{align*}
and since both sides are holomorphic on $(D\times D)^\times$, this equality holds on $(D\times D)^\times$. 

(2) By the assumption of Theorem \ref{main_thm2}, we have 
\begin{align*}
\cF^{\sigma,\tau}_{\rho,\rK\uparrow}(\cdot,0)|_{V_\rho}&\in\Hom_{\widetilde{K}}\bigl(V_\rho,\cP(\fp^+,V_\sigma\otimes V_\tau)[\det^{-1}]\bigr)=\Hom_{\widetilde{K}}\bigl(V_\rho,\cP(\fp^+,V_\sigma\otimes V_\tau)\bigr), 
\end{align*}
and hence we get $(\cF^{\sigma,\tau}_{\rho,\rK\uparrow}v)(x,0)\in\cP(\fp^+,V_\sigma\otimes V_\tau)$, and $(\cF^{\sigma,\tau}_{\rho,\rK\uparrow}v)(x,y)=(\cF^{\sigma,\tau}_{\rho,\rK\uparrow}v)(x-y,0)\allowbreak \in\cP(\fp^+\oplus\fp^+,V_\sigma\otimes V_\tau)$ holds. 
\end{proof}

To prove $\cF^{\sigma,\tau}_{\rho,\rK\uparrow}(\cO(D,V_\rho))\subset \cO(D\times D,V_\sigma\otimes V_\tau)$, 
we use the following lemma, which holds for general complex manifolds. 

\begin{lemma}
Let $X$ be a complex manifold, and let $Y\subset X$ be a closed complex submanifold such that $\dim Y<\dim X$ holds. 
Then $\cO(X)\subset\cO(X\setminus Y)$ is a closed subspace, and the topology on $\cO(X)$ coincides with the relative topology induced from $\cO(X\setminus Y)$. 
\end{lemma}

\begin{proof}
Suppose $\dim X=n$, $\dim Y=m$. For any point $p\in Y$, we can take a relatively compact open neighborhood $U\subset X$ of $p$ and a holomorphic coordinate $\varphi\colon U\overset{\sim}{\to}U_1\times U_2\subset\BC^m\times \BC^{n-m}$ such that $Y\cap U=\varphi^{-1}(U_1\times \{0\})$ holds. 
For $\varepsilon>0$, let $B_0(\varepsilon)\subset\BC^{n-m}$ be the closed ball with radius $\varepsilon$ centered at the origin. 
We take a sufficiently small $\varepsilon>0$ such that $B_0(\varepsilon)\subset U_2$ holds, and let $U(\varepsilon):=\varphi^{-1}(U_1\times(U_2\setminus B_0(\varepsilon)))$, so that $U(\varepsilon)\subset X\setminus Y$ is relatively compact. 
Then for $f\in\cO(X)$, by the maximum modulus principle, we have 
\begin{align*}
\sup_{x\in U}|f(x)|=\sup_{(x_1,x_2)\in U_1\times\partial U_2}|f(\varphi^{-1}(x_1,x_2))|=\sup_{x\in U(\varepsilon)}|f(x)|. 
\end{align*}
Next, we take any relatively compact open set $K\subset X$. 
Then since $K$ is covered by a finite number of coordinate neighborhoods around $Y$ and a relatively compact set in $X\setminus Y$, 
we can find a relatively compact set $K'\subset X\setminus Y$ such that 
$\sup_{x\in K}|f(x)|\le \sup_{x\in K'}|f(x)|$ holds for any $f\in\cO(X)$. 
Hence the topology on $\cO(X)$ coincides with the relative topology induced from $\cO(X\setminus Y)$. 

Next, suppose $f\in\cO(X\setminus Y)$ lies in the closure of $\cO(X)$, and we take a sequence $\{f_j\}\subset\cO(X)$ converging to $f$ in the topology of $\cO(X\setminus Y)$. 
Then $\{f_j\}$ is a Cauchy sequence in $\cO(X\setminus Y)$, and by the previous discussion, this is also a Cauchy sequence in the topology on $\cO(X)$, and thus converges to some $f_\infty\in\cO(X)$. Then we have $f=f_\infty$ on $X\setminus Y$, and hence we get $f\in\cO(X)$. 
Therefore $\cO(X)\subset\cO(X\setminus Y)$ is closed. 
\end{proof}

For $j\in\{0,\ldots,r\}$, let 
\[ (D\times D)_j:=\{(x,y)\in D\times D\mid\rank(x-y)\ge j\}, \]
so that $(D\times D)_j\setminus(D\times D)_{j+1}\subset (D\times D)_{j}$ is a closed complex submanifold. Then by applying the above lemma to 
\[ \cO(D\times D)\subset\cO((D\times D)_1)\subset\cdots\subset\cO((D\times D)_r)=\cO((D\times D)^\times), \]
$\cO(D\times D)\subset\cO((D\times D)^\times)$ is a closed subspace, and the topology on $\cO(D\times D)$ coincides with the relative topology induced from $\cO((D\times D)^\times)$. 

\begin{proof}[Proof of Theorem \ref{main_thm2}]
By the assumption on the irreducibility of $(\hat{\rho},\cO(D,V_\rho))$, we have 
\[ \cO(D,V_\rho)_{\widetilde{K}}=d\hat{\rho}(\cU(\fg^\BC))V_\rho, \]
where $\cU(\fg^\BC)$ is the universal enveloping algebra of $\fg^\BC$, and by Lemmas \ref{lem_holomorphic} and \ref{lem_const_func}, we have 
\[ \cF^{\sigma,\tau}_{\rho,\rK\uparrow}(\cO(D,V_\rho)_{\widetilde{K}})
=\cF^{\sigma,\tau}_{\rho,\rK\uparrow}(d\hat{\rho}(\cU(\fg))V_\rho)
=d(\hat{\sigma}\otimes\hat{\tau})\cF^{\sigma,\tau}_{\rho,\rK\uparrow}(V_\rho)
\subset\cO(D\times D,V_\sigma\otimes V_\tau). \]
Then since $\cO(D,V_\rho)_{\widetilde{K}}\subset\cO(D,V_\rho)$ is dense, $\cF^{\sigma,\tau}_{\rho,\rK\uparrow}\colon\cO(D,V_\rho)\to\cO((D\times D)^\times,V_\sigma\otimes V_\tau)$ is continuous and $\cO(D\times D,V_\sigma\otimes V_\tau)\subset\cO((D\times D)^\times,V_\sigma\otimes V_\tau)$ is closed, 
we get $\cF^{\sigma,\tau}_{\rho,\rK\uparrow}(\cO(D,V_\rho))\subset\cO(D\times D,V_\sigma\otimes V_\tau)$. 
\end{proof}

\section{Computation of images of minimal $\widetilde{K}$-types when $\sigma,\tau$ are one-dimensional}\label{section_minKtype}

In this section, we consider the case $\sigma=\chi^{-\lambda}$, $\tau=\chi^{-\mu}$ are one-dimensional. As before, let $\rank\fp^+=:r$, $\dim\fp^+=:n=r+\frac{d}{2}r(r-1)$. 
We recall from Corollary \ref{cor_scalar1} that, 
for $\lambda,\mu\in\BC$, $\bk\in\BZ_{++}^r$ with $\Re\lambda, \Re\mu>-k_r+\frac{n}{r}-1$ and for a fixed polynomial 
\[ \rK_\bk(x)\in\bigl(\cP(\fp^+)\otimes \Hom_\BC(V_{\bk\gamma}^\vee,\BC)\bigr)^K, \]
where $(\rho_{\bk\gamma}^\vee,V_{\bk\gamma}^\vee)$ is the irreducible representation of $K$ with the lowest weight $-(k_1\gamma_1+\cdots+k\gamma_r)\in(\ft^\BC)^\vee$, the map 
\begin{gather}
\cF^{\lambda,\mu}_{\bk\uparrow}\colon\cO_{\lambda+\mu}(D,V_{\bk\gamma}^\vee)\longrightarrow\cO_\lambda(D)\hotimes\cO_\mu(D), \notag\\
\begin{split}
(\cF^{\lambda,\mu}_{\bk\uparrow} f)(x,y):=\det(x-y)^{-\lambda-\mu+\frac{n}{r}}\int_{C(x,y)} \det(w-y)^{\lambda-\frac{n}{r}}\det(x-w)^{\mu-\frac{n}{r}} \\
{}\times \rK_\bk\bigl(((w-y)^{-1}+(x-w)^{-1})^{-1}\bigr) f(w)\,dw 
\end{split}\label{formula_HO1}
\end{gather}
becomes a holographic operator, which is unique up to a constant multiple. 
Hence the normalization of holographic operators is determined by the images of minimal $\widetilde{K}$-types, i.e., constant functions. 
The goal of this section is to compute $\cF^{\lambda,\mu}_{\bk\uparrow}v$ for constant functions $v\in V_{\bk\gamma}^\vee$. 

\begin{theorem}\label{thm_minKtype}
Let $\lambda,\mu\in\BC$, $\bk\in\BZ_{++}^r$, $\Re\lambda, \Re\mu>-k_r+\frac{n}{r}-1$. 
For constant functions $v\in V_{\bk\gamma}^\vee$, we have 
\[ (\cF^{\lambda,\mu}_{\bk\uparrow} v)(x,y)=B_r(\lambda,\mu,\bk)\rK_\bk(x-y)v, \]
where 
\begin{multline}
B_r(\lambda,\mu,\bk):=(2\pi)^{dr(r-1)/4} \\
{}\times\prod_{j=1}^r\frac{\Gamma\bigl(\lambda+k_j-\frac{d}{2}(j-1)\bigr)
\Gamma\bigl(\mu+k_j-\frac{d}{2}(j-1)\bigr)}{\Gamma\bigl(\lambda+\mu-\frac{d}{2}(j-1)\bigr)}
\frac{\prod_{1\le i<j\le r}\bigl(\lambda+\mu-\frac{d}{2}(i+j-1)\bigr)_{k_i+k_j}}
{\prod_{1\le i\le j\le r}\bigl(\lambda+\mu-\frac{d}{2}(i+j-2)\bigr)_{k_i+k_j}}. \label{Beta_vect}
\end{multline}
\end{theorem}

Here, $(\lambda)_k:=\lambda(\lambda+1)\cdots(\lambda+k-1)$. 
We recall another holographic operator from Theorem \ref{thm_int_expr_D}, 
\begin{gather}
\tilde{\cF}^{\lambda,\mu}_{\bk\uparrow}\colon\cH_{\lambda+\mu}(D,V_{\bk\gamma}^\vee)\longrightarrow \cH_\lambda(D)\hotimes\cH_\mu(D), \notag\\
(\tilde{\cF}^{\lambda,\mu}_{\bk\uparrow}f)(x,y):=\langle f,\hat{\rK}^{\lambda,\mu}_\bk(x,y;\overline{\cdot})^*\rangle_{\cH_{\lambda+\mu}(D,V_{\bk\gamma}^\vee)}, \label{formula_HO2}
\end{gather}
where 
\begin{align*}
\hat{\rK}^{\lambda,\mu}_\bk(x,y;\overline{w}):=h(x,\overline{w})^{-\lambda}h(y,\overline{w})^{-\mu}\rK_\bk(x^{\overline{w}}-y^{\overline{w}})
\in\cO(D\times D\times\overline{D},\Hom_\BC(V_{\bk\gamma}^\vee,\BC)). 
\end{align*}
If $f(w)=v\in V_{\bk\gamma}^\vee\subset\cH_{\lambda+\mu}(D,V_{\bk\gamma}^\vee)$ is a constant function, then since $v$ is orthogonal to homogeneous terms of degree $\ge 1$, 
by the normalization of the inner product, we have 
\begin{align*}
(\tilde{\cF}^{\lambda,\mu}_{\bk\uparrow}v)(x,y)
&=\langle v,\hat{\rK}^{\lambda,\mu}_\bk(x,y;\overline{\cdot})^*\rangle_{\cH_{\lambda+\mu}(D,V_{\bk\gamma}^\vee)}
=(v,\hat{\rK}^{\lambda,\mu}_\bk(x,y;0)^*)_{V_{\bk\gamma}^\vee} \\
&=(v,\rK_\bk(x-y)^*)_{V_{\bk\gamma}^\vee}=\rK_\bk(x-y)v. 
\end{align*}
Hence by Theorem \ref{thm_minKtype}, we immediately get the following. 

\begin{corollary}
Let $\lambda,\mu\in\BR$, $\bk\in\BZ_{++}^r$, $\lambda, \mu>\frac{2n}{r}-1$. 
The two holographic operators 
\[ \cF^{\lambda,\mu}_{\bk\uparrow},\tilde{\cF}^{\lambda,\mu}_{\bk\uparrow}\colon \cH_{\lambda+\mu}(D,V_{\bk\gamma}^\vee)\longrightarrow \cH_\lambda(D)\hotimes\cH_\mu(D) \]
in (\ref{formula_HO1}), (\ref{formula_HO2}) are related as 
\[ \cF^{\lambda,\mu}_{\bk\uparrow} =B_r(\lambda,\mu,\bk)\tilde{\cF}^{\lambda,\mu}_{\bk\uparrow}, \]
where $B_r(\lambda,\mu,\bk)$ is as in (\ref{Beta_vect}). 
\end{corollary}

In the rest of this section, we give a proof of Theorem \ref{thm_minKtype}. First, we prepare some notations. 
For $\bk\in\BZ_{++}^r$, let $\cP_\bk(\fp^+)\subset\cP(\fp^+)$ be the irreducible subrepresentation of $K^\BC$ isomorphic to $V_{\bk\gamma}^\vee$, so that 
\[ \cP(\fp^+)=\bigoplus_{\bk\in\BZ_{++}^r}\cP_\bk(\fp^+) \simeq\bigoplus_{\bk\in\BZ_{++}^r}V_{\bk\gamma}^\vee, \]
and let $\Delta_\bk(z), \check{\Delta}_\bk(z)\in\cP_\bk(\fp^+)$ be the lowest and highest weight vectors respectively such that $\Delta_\bk(e)=\check{\Delta}_\bk(e)=1$ holds. 
Let $M_LA_LN_L, M_LA_LN_L^-\subset L$ be the minimal parabolic subgroups as in Section \ref{subsection_root}. Then $A_LN_L, A_LN_L^-$ acts simply transitively on $\Omega$, and we have 
\begin{align*}
\Delta_\bk(b.z)&=\Delta_\bk(b.e)\Delta_\bk(z) && (z\in\fp^+,\;b\in A_LN_L^-), \\ 
\check{\Delta}_\bk(b.z)&=\check{\Delta}_\bk(b.e)\check{\Delta}_\bk(z) && (z\in\fp^+,\;b\in A_LN_L).  
\end{align*}
Also, for $k_0\in\BZ_{\ge 0}$, $\bk\in\BZ_{++}^r$, let $\underline{k_0}:=(\overbrace{k_0,\ldots,k_0}^r)$, $\bk^\vee:=(k_r,k_{r-1},\ldots,k_1)$. 
Then for $k_0\ge k_1$, we have 
\begin{equation}\label{formula_Delta_inv}
\det(z)^{k_0}\Delta_\bk(z^{-1})=\check{\Delta}_{\underline{k_0}-\bk^\vee}(z). 
\end{equation}
For $\lambda\in\BC$, $\bk\in\BC^r$, $\bm\in(\BZ_{\ge 0})^r$, let 
\begin{align}
\Gamma_r(\lambda+\bk)&:=(2\pi)^{dr(r-1)/4}\prod_{j=1}^r \Gamma\biggl(\lambda+k_j-\frac{d}{2}(j-1)\biggr), \label{Gamma}\\
(\lambda+\bk)_\bm&:=\frac{\Gamma_r(\lambda+\bk+\bm)}{\Gamma_r(\lambda+\bk)}
=\prod_{j=1}^r\biggl(\lambda+k_j-\frac{d}{2}(j-1)\biggr)_{m_j}, \label{Pochhammer}
\end{align}
and let $\Gamma_r(\lambda):=\Gamma_r(\lambda+(0,\ldots,0))$, $(\lambda)_\bm:=(\lambda+(0,\ldots,0))_\bm$, so that we have 
\[ (\lambda+\bk)_\bm=(-1)^{|\bm|}\biggl(-\lambda+\frac{n}{r}-\bk^\vee-\bm^\vee\biggr)_{\bm^\vee}. \]
If $\Re\lambda>-k_r+\frac{n}{r}-1$, $\bk\in\BZ_{++}^r$, $f(x)\in\cP_\bk(\fp^+)$, then by \cite[Theorem VII.1.1, Lemma XI.2.3]{FK}, we have 
\begin{align}
\int_\Omega e^{-\tr(x)}f(x)\det(x)^{\lambda-\frac{n}{r}}\,dx=\Gamma_r(\lambda+\bk)f(e). \label{formula_Gamma}
\end{align}
Using $\Gamma_r(\lambda+\bk)$, (\ref{Beta_vect}) is rewritten as 
\[ B_r(\lambda,\mu,\bk)=
\frac{\Gamma_r(\lambda+\bk)\Gamma_r(\mu+\bk)}{\Gamma_r(\lambda+\mu)}
\frac{\prod_{1\le i<j\le r}\bigl(\lambda+\mu-\frac{d}{2}(i+j-1)\bigr)_{k_i+k_j}}
{\prod_{1\le i\le j\le r}\bigl(\lambda+\mu-\frac{d}{2}(i+j-2)\bigr)_{k_i+k_j}}. \]
Especially, if $\bk=(l,\ldots,l)$, then we have 
\begin{align*}
B_r(\lambda,\mu,(l,\dots,l))
&=\frac{\Gamma_r(\lambda+l)\Gamma_r(\mu+l)}{\Gamma_r(\lambda+\mu)}
\frac{\prod_{1\le i<j\le r}\bigl(\lambda+\mu-\frac{d}{2}(i+j-1)\bigr)_{2l}}
{\prod_{1\le i\le j\le r}\bigl(\lambda+\mu-\frac{d}{2}(i+j-2)\bigr)_{2l}} \\
&=\frac{\Gamma_r(\lambda+l)\Gamma_r(\mu+l)}{\Gamma_r(\lambda+\mu)}
\frac{\prod_{1\le i<j\le r}\bigl(\lambda+\mu-\frac{d}{2}(i+j-1)\bigr)_{2l}}
{\prod_{0\le i< j\le r}\bigl(\lambda+\mu-\frac{d}{2}(i+j-1)\bigr)_{2l}} \\
&=\frac{\Gamma_r(\lambda+l)\Gamma_r(\mu+l)}{\Gamma_r(\lambda+\mu)}
\frac{1}{\prod_{j=1}^r\bigl(\lambda+\mu-\frac{d}{2}(j-1)\bigr)_{2l}} \\
&=\frac{\Gamma_r(\lambda+l)\Gamma_r(\mu+l)}{\Gamma_r(\lambda+\mu)(\lambda+\mu)_{\underline{2l}}}
=\frac{\Gamma_r(\lambda+l)\Gamma_r(\mu+l)}{\Gamma_r(\lambda+\mu+2l)}, 
\end{align*}
which coincides with the beta function in \cite[Theorem VII.1.7]{FK}. 

Next, for $\alpha>\frac{2n}{r}-1$, let $\langle f_1,f_2\rangle_\alpha=\langle f_1(z),f_2(z)\rangle_{\alpha,z}$ be the $\widetilde{G}$-invariant inner product on $\cH_\alpha(D)\subset\cO_\alpha(D)$ given in (\ref{inner_prod}) with respect to the variable $z$. Then the following formulas hold. 

\begin{lemma}[{Faraut--Kor\'anyi \cite{FK0}, \cite[Corollary XIII.2.3]{FK}}]\label{lem_FK}
For $\alpha>\frac{2n}{r}-1$, $\bm\in\BZ_{++}^r$ and $f(z)\in\cP_\bm(\fp^+)$, we have 
\[ \bigl\langle f(z),e^{(z|\overline{x})}\bigr\rangle_{\alpha,z}=\frac{1}{(\alpha)_\bm}f(x). \]
\end{lemma}

\begin{lemma}[{\cite[Theorem 8.4]{N3}}]\label{lem_Plancherel}
For $\alpha,\beta>\frac{2n}{r}-1$, $\bl\in\BZ_{++}^r$, $f(z)\in\cP_\bl(\fp^+)$, we have 
\begin{align*}
&\bigl\langle f(w+z),e^{(w+z|\overline{x})}\bigr\rangle_{(\alpha,w)\otimes(\beta,z)}
=\frac{1}{(\alpha)_\bl(\beta)_\bl}\frac{\prod_{1\le i\le j\le r}\bigl(\alpha+\beta-1-\frac{d}{2}(i+j-2)\bigr)_{l_i+l_j}}{\prod_{1\le i<j\le r+1}\bigl(\alpha+\beta-1-\frac{d}{2}(i+j-3)\bigr)_{l_i+l_j}}f(x), 
\end{align*}
where $l_{r+1}:=0$. 
\end{lemma}

Especially, for $f(z)\in \cP_\bl(\fp^+)$, we consider the orthogonal decomposition of $f(w+z)\in\cP(\fp^+\oplus\fp^+)$ as 
\[ f(w+z)=\sum_{\bm,\bn\in\BZ_{++}^r}f_{\bm,\bn}(w,z)\in\bigoplus_{\bm,\bn\in\BZ_{++}^r}\cP_{\bm}(\fp^+)\otimes\cP_{\bn}(\fp^+). \]
Then we have 
\begin{align}
&\bigl\langle f(w+z),e^{(w+z|\overline{x})}\bigr\rangle_{(\alpha,w)\otimes(\beta,z)}=\sum_{\bm,\bn}\frac{1}{(\alpha)_\bm(\beta)_\bn}f_{\bm,\bn}(x,x) \notag\\
&=\frac{1}{(\alpha)_\bl(\beta)_\bl}\frac{\prod_{1\le i\le j\le r}\bigl(\alpha+\beta-1-\frac{d}{2}(i+j-2)\bigr)_{l_i+l_j}}{\prod_{1\le i<j\le r+1}\bigl(\alpha+\beta-1-\frac{d}{2}(i+j-3)\bigr)_{l_i+l_j}}f(x). \label{formula_inner_eq}
\end{align}

\begin{proof}[Proof of Theorem \ref{thm_minKtype}]
By Lemma \ref{lem_const_func}, $(\cF^{\lambda,\mu}_{\bk\uparrow}v)(x,0)$ is holomorphically continued to $x\in\fp^+$, and for $(x,y)\in (D\times D)^\times$, we have 
\[ (\cF^{\lambda,\mu}_{\bk\uparrow}v)(x,y)=(\cF^{\lambda,\mu}_{\bk\uparrow}v)(x-y,0). \]
Since this is holomorphic, it is enough to consider the case $x\in\Omega$ and $y=0$. 
Next, for fixed $\rK_\bk(x)\in\bigl(\cP_\bk(\fp^+)\otimes\Hom_\BC(V_{\bk\gamma}^\vee,\BC)\bigr)^K$, let $v_0\in V_{\bk\gamma}^\vee$ be the lowest weight vector satisfying 
\[ \rK_\bk(x)v_0=\Delta_\bk(x). \]
Then since $\cF^{\lambda,\mu}_{\bk\uparrow}|_{V_{\bk\gamma}^\vee}$ intertwines the $\widetilde{K}^\BC$-action, we may assume $v=v_0$, so that 
\begin{align*}
&(\cF^{\lambda,\mu}_{\bk\uparrow} v_0)(x,0) \\
&=\det(x)^{-\lambda-\mu+\frac{n}{r}}\int_{\Omega\cap(x-\Omega)} 
\det(w)^{\lambda-\frac{n}{r}}\det(x-w)^{\mu-\frac{n}{r}} 
\Delta_\bk\bigl((w^{-1}+(x-w)^{-1})^{-1}\bigr)\,dw . 
\end{align*}
Let $x\in\Omega$. Since $A_LN_L^-$ acts simply transitively on $\Omega$, we can take $b\in A_LN_L^-$ such that $b.e=x$ holds. 
Then since $\cF^{\lambda,\mu}_{\bk\uparrow}|_{V_{\bk\gamma}^\vee}$ intertwines the action of $A_LN_L^-\subset L$, we have 
\begin{align}
(\cF^{\lambda,\mu}_{\bk\uparrow}v_0)(x,0)&=(\cF^{\lambda,\mu}_{\bk\uparrow}v_0)(b.e,0)
=\chi(b)^{-\lambda-\mu}(\cF^{\lambda,\mu}_{\bk\uparrow}\circ (\chi^{-\lambda-\mu}\otimes \rho_{\bk\gamma}^\vee) (b^{-1})v_0)(e,0) \notag\\
&=(\cF^{\lambda,\mu}_{\bk\uparrow}v_0)(e,0)\Delta_\bk(b.e)
=(\cF^{\lambda,\mu}_{\bk\uparrow}v_0)(e,0)\Delta_\bk(x). \label{formula_AN_equiv}
\end{align}
Hence it is enough to compute 
\begin{align*}
(\cF^{\lambda,\mu}_{\bk\uparrow} v_0)(e,0)
=\int_{\Omega\cap(e-\Omega)} \det(w)^{\lambda-\frac{n}{r}}\det(e-w)^{\mu-\frac{n}{r}}\Delta_\bk\bigl((w^{-1}+(e-w)^{-1})^{-1}\bigr)\,dw. 
\end{align*}
We fix an arbitrary $k_0\in\BZ_{\ge 0}$ such that $k_0\ge k_1$. Then by (\ref{formula_Gamma}), (\ref{formula_AN_equiv}), we have 
\begin{align*}
&\Gamma_r(\lambda+\mu+k_0+\bk)(\cF^{\lambda,\mu}_{\bk\uparrow} v_0)(e,0) \\
&=\int_{x\in\Omega}e^{-\tr(x)}\Delta_\bk(x)\det(x)^{\lambda+\mu+k_0-\frac{n}{r}}(\cF^{\lambda,\mu}_{\bk\uparrow} v_0)(e,0)\,dx \\
&=\int_{x\in\Omega}e^{-\tr(x)}\det(x)^{\lambda+\mu+k_0-\frac{n}{r}}(\cF^{\lambda,\mu}_{\bk\uparrow} v_0)(x,0)\,dx \\
&=\int_{x\in\Omega}e^{-\tr(x)}\det(x)^{k_0}
\int_{w\in\Omega\cap(x-\Omega)}\det(w)^{\lambda-\frac{n}{r}}\det(x-w)^{\mu-\frac{n}{r}} \\* 
&\hspace{240pt}\times\Delta_\bk\bigl((w^{-1}+(x-w)^{-1})^{-1}\bigr)\,dwdx \\
&=\iint_{(w,z)\in\Omega\times\Omega}e^{-\tr(w+z)}\det(w+z)^{k_0}\det(w)^{\lambda-\frac{n}{r}}\det(z)^{\mu-\frac{n}{r}}\Delta_\bk\bigl((w^{-1}+z^{-1})^{-1}\bigr)\,dwdz \\
&=\iint_{(w,z)\in\Omega\times\Omega}e^{-\tr(w+z)}\det(w^{-1}+z^{-1})^{k_0} \\*
&\hspace{107pt}\times\det(w)^{\lambda+k_0-\frac{n}{r}}\det(z)^{\mu+k_0-\frac{n}{r}}\Delta_\bk\bigl((w^{-1}+z^{-1})^{-1}\bigr)\,dwdz \\
&=\iint_{(w,z)\in\Omega\times\Omega}e^{-\tr(w+z)}\det(w)^{\lambda+k_0-\frac{n}{r}}\det(z)^{\mu+k_0-\frac{n}{r}}\check{\Delta}_{\underline{k_0}-\bk^\vee}(w^{-1}+z^{-1})\,dwdz, 
\end{align*}
where we have put $x=w+z$ at the 4th equality, used (\ref{formula_det_inv2}) at the 5th equality, and used (\ref{formula_Delta_inv}) at the 6th equality. 
For $\bm,\bn\in\BZ_{++}^r$, let $f_{\bm,\bn}(w,z)$ be the orthogonal projection of $\check{\Delta}_{\underline{k_0}-\bk^\vee}(w+z)$ 
onto $\cP_{\bm}(\fp^+)\otimes\cP_{\bn}(\fp^+)$, namely, 
\[ \check{\Delta}_{\underline{k_0}-\bk^\vee}(w+z)=\sum_{\bm,\bn\in\BZ_{++}^r}f_{\bm,\bn}(w,z)\in\bigoplus_{\bm,\bn\in\BZ_{++}^r}\cP_{\bm}(\fp^+)\otimes\cP_{\bn}(\fp^+). \]
Then $f_{\bm,\bn}\ne 0$ holds only if $|\bm|+|\bn|=|\underline{k_0}-\bk^\vee|$, and by (\ref{formula_Gamma}), we have 
\begin{align*}
&\Gamma_r(\lambda+\mu+k_0+\bk)(\cF^{\lambda,\mu}_{\bk\uparrow} v_0)(e,0) \\
&=\sum_{\bm,\bn}\iint_{(w,z)\in\Omega\times\Omega}e^{-\tr(w)}e^{-\tr(z)}\det(w)^{\lambda+k_0-\frac{n}{r}}\det(z)^{\mu+k_0-\frac{n}{r}}f_{\bm,\bn}(w^{-1},z^{-1})\,dwdz \\
&=\sum_{\bm,\bn}\Gamma_r(\lambda+k_0-\bm^\vee)\Gamma_r(\mu+k_0-\bn^\vee)f_{\bm,\bn}(e,e) \\
&=\sum_{\bm,\bn}\frac{\Gamma_r(\lambda+k_0)\Gamma_r(\mu+k_0)}{(\lambda+k_0-\bm^\vee)_{\bm^\vee}(\mu+k_0-\bn^\vee)_{\bn^\vee}}f_{\bm,\bn}(e,e) \\
&=(-1)^{|\underline{k_0}-\bk^\vee|}\Gamma_r(\lambda+k_0)\Gamma_r(\mu+k_0)
\sum_{\bm,\bn}\frac{f_{\bm,\bn}(e,e)}{\bigl(-\lambda-k_0+\frac{n}{r}\bigr)_\bm\bigl(-\mu-k_0+\frac{n}{r}\bigr)_\bn}. 
\end{align*}
Then applying (\ref{formula_inner_eq}) for $\bl=\underline{k_0}-\bk^\vee$, $\alpha=-\lambda-k_0+\frac{n}{r}$, $\beta=-\mu-k_0+\frac{n}{r}$, we get 
\begin{align*}
&(\cF^{\lambda,\mu}_{\bk\uparrow} v_0)(e,0) \\
&=(-1)^{|\underline{k_0}-\bk^\vee|}\frac{\Gamma_r(\lambda+k_0)\Gamma_r(\mu+k_0)}{\Gamma_r(\lambda+\mu+k_0+\bk)}
\frac{\check{\Delta}_{\underline{k_0}-\bk^\vee}(e)}{\bigl(-\lambda-k_0+\frac{n}{r}\bigr)_{\underline{k_0}-\bk^\vee}\bigl(-\mu-k_0+\frac{n}{r}\bigr)_{\underline{k_0}-\bk^\vee}} \\*
&\eqspace{}\times\frac{\prod_{1\le i\le j\le r}\bigl(-\lambda-\mu-2k_0+d(r-1)+1-\frac{d}{2}(i+j-2)\bigr)_{2k_0-k_{r-i+1}-k_{r-j+1}}}
{\prod_{1\le i<j\le r+1}\bigl(-\lambda-\mu-2k_0+d(r-1)+1-\frac{d}{2}(i+j-3)\bigr)_{2k_0-k_{r-i+1}-k_{r-j+1}}} \\
&=\frac{\Gamma_r(\lambda+k_0)\Gamma_r(\mu+k_0)}{\Gamma_r(\lambda+\mu+k_0+\bk)}
\frac{1}{(\lambda+\bk)_{\underline{k_0}-\bk}(\mu+\bk)_{\underline{k_0}-\bk}} \\*
&\eqspace{}\times\frac{\prod_{1\le i\le j\le r}\bigl(\lambda+\mu+k_i+k_j-\frac{d}{2}(i+j-2)\bigr)_{2k_0-k_i-k_j}}
{\prod_{0\le i<j\le r}\bigl(\lambda+\mu+k_i+k_j-\frac{d}{2}(i+j-1)\bigr)_{2k_0-k_i-k_j}}  \\
&=\frac{\Gamma_r(\lambda+\bk)\Gamma_r(\mu+\bk)}{\Gamma_r(\lambda+\mu+k_0+\bk)}
\prod_{1\le i\le j\le r}\frac{\bigl(\lambda+\mu-\frac{d}{2}(i+j-2)\bigr)_{2k_0}}{\bigl(\lambda+\mu-\frac{d}{2}(i+j-2)\bigr)_{k_i+k_j}} \\*
&\eqspace{}\times\prod_{0\le i<j\le r}\frac{\bigl(\lambda+\mu-\frac{d}{2}(i+j-1)\bigr)_{k_i+k_j}}{\bigl(\lambda+\mu-\frac{d}{2}(i+j-1)\bigr)_{2k_0}} \\
&=\frac{\Gamma_r(\lambda+\bk)\Gamma_r(\mu+\bk)}{\Gamma_r(\lambda+\mu)(\lambda+\mu)_{\underline{k_0}+\bk}}
\frac{\prod_{1\le i\le j\le r}\bigl(\lambda+\mu-\frac{d}{2}(i+j-2)\bigr)_{2k_0}}{\prod_{1\le i\le j\le r}\bigl(\lambda+\mu-\frac{d}{2}(i+j-2)\bigr)_{k_i+k_j}} \\*
&\eqspace{}\times\frac{\prod_{j=1}^r\bigl(\lambda+\mu-\frac{d}{2}(j-1)\bigr)_{k_0+k_j}\prod_{1\le i<j\le r}\bigl(\lambda+\mu-\frac{d}{2}(i+j-1)\bigr)_{k_i+k_j}}{\prod_{1\le i\le j\le r}\bigl(\lambda+\mu-\frac{d}{2}(i+j-2)\bigr)_{2k_0}} \\
&=\frac{\Gamma_r(\lambda+\bk)\Gamma_r(\mu+\bk)}{\Gamma_r(\lambda+\mu)}
\frac{\prod_{1\le i<j\le r}\bigl(\lambda+\mu-\frac{d}{2}(i+j-1)\bigr)_{k_i+k_j}}{\prod_{1\le i\le j\le r}\bigl(\lambda+\mu-\frac{d}{2}(i+j-2)\bigr)_{k_i+k_j}}
=B_r(\lambda,\mu,\bk), 
\end{align*}
and this completes the proof of Theorem \ref{thm_minKtype}. 
\end{proof}

\section{Differential expression of holographic operators when $\sigma,\tau,\allowbreak \rho$ are one-dimensional}\label{section_diff}

In this section, we assume that $\sigma=\chi^{-\lambda}$, $\tau=\chi^{-\mu}$, $\rho=\chi^{-(\lambda+\mu+2l)}$ are one-dimensional. Then the holographic operator is given by 
\begin{gather*}
\tilde{\cF}^{\lambda,\mu}_{l\uparrow}\colon\cH_{\lambda+\mu+2l}(D)\longrightarrow \cH_\lambda(D)\hotimes\cH_\mu(D), \\
\begin{split}
&(\tilde{\cF}^{\lambda,\mu}_{l\uparrow}f)(x,y):=\bigl\langle f,\overline{\hat{\rK}^{\lambda,\mu}_l(x,y;\overline{\cdot})}\bigr\rangle_{\lambda+\mu+2l} \\
&=\frac{\det(x-y)^{-\lambda-\mu-l+\frac{n}{r}}}{B_r(\lambda,\mu,(l,\ldots,l))} \int_{C(x,y)} \det(w-y)^{\lambda+l-\frac{n}{r}}\det(x-w)^{\mu+l-\frac{n}{r}}f(w)\,dw 
\end{split}
\end{gather*}
where 
\begin{gather*}
\hat{\rK}^{\lambda,\mu}_l(x,y;\overline{w}):=h(x,\overline{w})^{-\lambda}h(y,\overline{w})^{-\mu}\det(x^{\overline{w}}-y^{\overline{w}})^l, \\
B_r(\lambda,\mu,(l,\ldots,l)):=\frac{\Gamma_r(\lambda+l)\Gamma_r(\mu+l)}{\Gamma_r(\lambda+\mu+2l)},
\end{gather*}
with $\Gamma_r(\lambda)$ given in (\ref{Gamma}). 
We note that the 1st expression is valid without the tube type assumption, but the 2nd expression is valid only for tube type cases. 
In this section, we find another expression of the holographic operator as an infinite-order differential operator. 

To describe the differential expression, for $\bm\in\BZ_{++}^r$, 
let $\tilde{\Phi}_\bm(x,y)\in\cP_\bm(\fp^+)\otimes\cP_\bm(\fp^-)$ be the orthogonal projection of $e^{(x|y)}$ onto $\cP_\bm(\fp^+)\otimes\cP_\bm(\fp^-)$, so that 
\[ e^{(x|y)}=\sum_{\bm\in\BZ_{++}^r}\tilde{\Phi}_\bm(x,y). \]
This is defined without the tube type assumption. Especially, if $\fp^+$ is of tube type, then we have 
\[ \tilde{\Phi}_\bm(x,y)=\frac{\dim\cP_\bm(\fp^+)}{\bigl(\frac{n}{r}\bigr)_\bm}\Phi_\bm(P(y^{1/2})x)
=\frac{\dim\cP_\bm(\fp^+)}{\bigl(\frac{n}{r}\bigr)_\bm}\Phi_\bm(P(x^{1/2})y), \]
where $\Phi_\bm(x)\in\cP_\bm(\fp^+)$ is the unique $K_L$-invariant polynomial in $\cP_\bm(\fp^+)$ satisfying $\Phi_\bm(e)=1$, 
and $(\alpha)_\bm$ is as in (\ref{Pochhammer}) (see \cite[Proposition XII.1.3\,(i)]{FK}). 
Using this, for $\alpha,\beta\in\BC$, let 
\[ {}_1F_1\biggl(\begin{matrix}\alpha\\\beta\end{matrix}\,;\,x,y\biggr)
:=\sum_{\bm\in\BZ_{++}^r}\frac{(\alpha)_\bm}{(\beta)_\bm}\tilde{\Phi}_\bm(x,y). \]
Then $\tilde{\cF}^{\lambda,\mu}_{l\uparrow}$ is expressed as an infinite-order differential operator as follows. 

\begin{theorem}\label{thm_1F1}
For $f(w)\in\cO_{\lambda+\mu+2l}(D)_{\widetilde{K}}=\cP(\fp^+)$, we have 
\begin{align*}
(\tilde{\cF}^{\lambda,\mu}_{l\uparrow}f)(x,y)
&=\det(x-y)^l {}_1F_1\biggl(\begin{matrix}\lambda+l\\ \lambda+\mu+2l\end{matrix}\,;\, x-y, \frac{\partial}{\partial w}\biggr)f(w)\biggr|_{w=y} \\
&=\det(x-y)^l {}_1F_1\biggl(\begin{matrix}\mu+l\\ \lambda+\mu+2l\end{matrix}\,;\, y-x, \frac{\partial}{\partial w}\biggr)f(w)\biggr|_{w=x}. 
\end{align*}
\end{theorem}

Indeed, by Theorem \ref{thm_diff_expr}, for $x,y\in D\subset\fp^+$, $\zeta\in\fp^-$, by putting 
\[ F^{\lambda,\mu}_{l\uparrow,\rL}(x;\zeta):=(\tilde{\cF}^{\lambda,\mu}_{l\uparrow}e^{(w|\zeta)})_w(x,0), \qquad 
F^{\lambda,\mu}_{l\uparrow,\rR}(y;\zeta):=(\tilde{\cF}^{\lambda,\mu}_{l\uparrow}e^{(w|\zeta)})_w(0,y), \]
we have 
\[ (\tilde{\cF}^{\lambda,\mu}_{l\uparrow}f)(x,y)
=F^{\lambda,\mu}_{l\uparrow,\rL}\biggl(x-y;\frac{\partial}{\partial w}\biggr)f(w) \biggr|_{w=y} 
=F^{\lambda,\mu}_{l\uparrow,\rR}\biggl(y-x;\frac{\partial}{\partial w}\biggr)f(w) \biggr|_{w=x}. \]
Hence it suffices to compute $F^{\lambda,\mu}_{l\uparrow,\rL}(x;\zeta)$, $F^{\lambda,\mu}_{l\uparrow,\rR}(y;\zeta)$. 

In the following, we give two different proofs of Theorem \ref{thm_1F1}. 
One proof is based on the integral expression on $D$, which is valid without the tube type assumption. 
By \cite[Part V, (J6.1)]{FKKLR}, we have 
\begin{align*}
F^{\lambda,\mu}_{l\uparrow,\rL}(x;\zeta)&=\bigl\langle e^{(w|\zeta)}, \overline{\hat{\rK}^{\lambda,\mu}_l(x,0;\overline{w})}\bigr\rangle_{\lambda+\mu+2l,w}
=\bigl\langle e^{(w|\zeta)}, \overline{h(x,\overline{w})^{-\lambda} \det(x^{\overline{w}})^l}\bigr\rangle_{\lambda+\mu+2l,w} \\
&=\bigl\langle e^{(w|\zeta)}, \overline{h(x,\overline{w})^{-\lambda-l} \det(x)^l} \bigr\rangle_{\lambda+\mu+2l,w} \\
&=\det(x)^l\bigl\langle e^{(w|\zeta)}, \overline{h(x,\overline{w})^{-\lambda-l}} \bigr\rangle_{\lambda+\mu+2l,w}, \\
F^{\lambda,\mu}_{l\uparrow,\rR}(y;\zeta)&=\bigl\langle e^{(w|\zeta)}, \overline{\hat{\rK}^{\lambda,\mu}_l(0,y;\overline{w})}\bigr\rangle_{\lambda+\mu+2l,w}
=\bigl\langle e^{(w|\zeta)}, \overline{h(y,\overline{w})^{-\mu} \det(-y^{\overline{w}})^l} \bigr\rangle_{\lambda+\mu+2l,w} \\
&=\bigl\langle e^{(w|\zeta)}, \overline{h(y,\overline{w})^{-\mu-l} \det(-y)^l} \bigr\rangle_{\lambda+\mu+2l,w} \\
&=\det(-y)^l\bigl\langle e^{(w|\zeta)}, \overline{h(y,\overline{w})^{-\mu-l}} \bigr\rangle_{\lambda+\mu+2l,w}. 
\end{align*}
Then by combining 
\[ h(x,\overline{w})^{-\lambda}=\sum_{\bm\in\BZ_{++}^r} (\lambda)_\bm\tilde{\Phi}_\bm(x,\overline{w}) \]
(see \cite{FK0}, \cite[Proposition XII.1.3\,(ii), Theorem XIII.2.4]{FK}) with Lemma \ref{lem_FK}, we get 
\begin{align*}
F^{\lambda,\mu}_{l\uparrow,\rL}(x;\zeta)&=\det(x)^l\bigl\langle e^{(w|\zeta)}, \overline{h(x,\overline{w})^{-\lambda-l}}\bigr\rangle_{\lambda+\mu+2l,w} \\
&=\det(x)^l\sum_{\bm\in\BZ_{++}^r}(\lambda+l)_\bm \bigl\langle e^{(w|\zeta)}, \overline{\tilde{\Phi}_\bm(x,\overline{w})}\bigr\rangle_{\lambda+\mu+2l,w} \\
&=\det(x)^l\sum_{\bm\in\BZ_{++}^r}\frac{(\lambda+l)_\bm}{(\lambda+\mu+2l)_\bm} \tilde{\Phi}_\bm(x,\zeta)
=\det(x)^l {}_1F_1\biggl(\begin{matrix}\lambda+l\\ \lambda+\mu+2l\end{matrix}\,;\,x,\zeta\biggr). 
\end{align*}
Similarly, we have $F^{\lambda,\mu}_{l\uparrow,\rR}(y;\zeta)=\det(-y)^l {}_1F_1\bigl(\begin{smallmatrix}\mu+l\\ \lambda+\mu+2l\end{smallmatrix};y,\zeta\bigr)$. 

Another proof of Theorem \ref{thm_1F1} is based on the integral expression on the totally real submanifold $C(x,y)$, which is valid only for tube type cases. We have 
\begin{align*}
F^{\lambda,\mu}_{l\uparrow,\rL}(x;\zeta)
&=\frac{\det(x)^{-\lambda-\mu-l+\frac{n}{r}}}{B_r(\lambda,\mu,(l,\ldots,l))} \int_{C(x,0)} \det(w)^{\lambda+l-\frac{n}{r}}\det(x-w)^{\mu+l-\frac{n}{r}}e^{(w|\zeta)}\,dw, \\
F^{\lambda,\mu}_{l\uparrow,\rR}(y;\zeta)&=\frac{\det(-y)^{-\lambda-\mu-l+\frac{n}{r}}}{B_r(\lambda,\mu,(l,\ldots,l))} \int_{C(0,y)} \det(w-y)^{\lambda+l-\frac{n}{r}}\det(-w)^{\mu+l-\frac{n}{r}}e^{(w|\zeta)}\,dw \\
&=\frac{\det(-e)^{l-\frac{n}{r}}\det(y)^{-\lambda-\mu-l+\frac{n}{r}}}{B_r(\lambda,\mu,(l,\ldots,l))} \int_{C(0,y)} \det(y-w)^{\lambda+l-\frac{n}{r}}\det(w)^{\mu+l-\frac{n}{r}}e^{(w|\zeta)}\,dw \\
&=\frac{(-1)^{lr}\det(y)^{-\lambda-\mu-l+\frac{n}{r}}}{B_r(\lambda,\mu,(l,\ldots,l))} \int_{C(y,0)} \det(y-w)^{\lambda+l-\frac{n}{r}}\det(w)^{\mu+l-\frac{n}{r}}e^{(w|\zeta)}\,dw, 
\end{align*}
where we have used Lemma \ref{lem_ori_contour}\,(2) at the last equality. 
Since $F^{\lambda,\mu}_{l\uparrow,\rL}(x;\zeta)$ is holomorphic, we may assume $x\in\Omega\cap D$. Then as in the proof of \cite[Proposition XV.1.4]{FK}, we have 
\begin{align*}
&F^{\lambda,\mu}_{l\uparrow,\rL}(x;\zeta)
=\frac{\det(x)^{-\lambda-\mu-l+\frac{n}{r}}}{B_r(\lambda,\mu,(l,\ldots,l))} \int_{\Omega\cap(x-\Omega)} \det(w)^{\lambda+l-\frac{n}{r}}\det(x-w)^{\mu+l-\frac{n}{r}}e^{(w|\zeta)}\,dw \\
&=\frac{\det(x)^{-\lambda-\mu-l+\frac{n}{r}}}{B_r(\lambda,\mu,(l,\ldots,l))} \int_{\Omega\cap(e-\Omega)} \det(P(x^{1/2})w)^{\lambda+l-\frac{n}{r}} \det(P(x^{1/2})(e-w))^{\mu+l-\frac{n}{r}}e^{(P(x)^{1/2}w|\zeta)} \\*
&\hspace{355pt}\times\det(x)^{\frac{n}{r}}\,dw \\
&=\frac{\det(x)^l}{B_r(\lambda,\mu,(l,\ldots,l))} \int_{\Omega\cap(e-\Omega)} \det(w)^{\lambda+l-\frac{n}{r}}\det(e-w)^{\mu+l-\frac{n}{r}}e^{(w|P(x)^{1/2}\zeta)}\,dw \\
&=\frac{\det(x)^l}{B_r(\lambda,\mu,(l,\ldots,l))}\sum_{\bm\in\BZ_{++}^r} \int_{\Omega\cap(e-\Omega)} \det(w)^{\lambda+l-\frac{n}{r}}\det(e-w)^{\mu+l-\frac{n}{r}} \tilde{\Phi}_\bm(w,P(x)^{1/2}\zeta)\,dw \\
&=\frac{\Gamma_r(\lambda+\mu+2l)\det(x)^l}{\Gamma_r(\lambda+l)\Gamma_r(\mu+l)} \sum_{\bm\in\BZ_{++}^r} \frac{\Gamma_r(\lambda+l+\bm)\Gamma_r(\mu+l)}{\Gamma_r(\lambda+\mu+2l+\bm)}\tilde{\Phi}_\bm(e,P(x)^{1/2}\zeta) \\
&=\det(x)^l\sum_{\bm\in\BZ_{++}^r} \frac{(\lambda+l)_\bm}{(\lambda+\mu+2l)_\bm} \tilde{\Phi}_\bm(x,\zeta)
=\det(x)^l {}_1F_1\biggl(\begin{matrix}\lambda+l\\ \lambda+\mu+2l\end{matrix}\,;\,x,\zeta\biggr), 
\end{align*}
where the 5th equality follows from \cite[Theorem VII.1.7]{FK}.  
Similar for $F^{\lambda,\mu}_{l\uparrow,\rR}(y;\zeta)$.

\end{document}